\documentclass[11pt,reqno]{article}

\usepackage[T1]{fontenc}
\usepackage[polish,english]{babel}
\usepackage[utf8]{inputenc}

\usepackage{amssymb,amsmath,amsfonts,amsthm,bm}
\usepackage{mathtools}
\usepackage{mathdots}
\usepackage{mathrsfs}
\usepackage{hyperref}
\usepackage{enumitem}

\usepackage{xcolor}
\usepackage{tikz}
\usetikzlibrary {arrows.meta,patterns.meta,bending,positioning}

\usepackage{subcaption}
\usepackage{caption}

\newtheorem{theorem}{Theorem}[section]
\newtheorem{proposition}{Proposition}[section]
\newtheorem{lemma}{Lemma}[section]

\newtheorem{definition}{Definition}[section]

\newtheorem{remark}{Remark}[section]

\def\br{\mathbb R}

\def\itemc{\itemindent=10pt\labelsep=8pt\labelwidth10pt\itemsep=4pt}

\def\ve{\varepsilon}

\def\dim{\text{\rm dim\,}}
\def\ker{\text{\rm Ker\,}}
\def\bbR{{\mathbb R}}
\def\bbZ{{\mathbb Z}}
\def\cV{{\mathcal V}}

\def\br{{\mathbb{R}}}

\def\br{\mathbb R}

\def\ve{\varepsilon}

\newtheorem{remark-definition}{Remark and Definition}[section]
\newtheorem{rem-not}{Remark and Notation}[section]

\begin{document}

\title{Global Hopf bifurcation and symmetric periodic solutions in multi-agent systems with pseudoneutral memory} 
 \author{Casey Crane}

\maketitle

\begin{abstract}
We study the emergence of symmetric oscillatory behavior in multi-agent systems where each agent incorporates distributed memory of past states and feedback from changes in its response over a finite time window. The latter admits a neutral-form presentation but reduces to a finite difference of delayed states. The closed-loop dynamics are therefore described by a system of nonlinear retarded functional differential equations with a high degree of symmetry, arising from a network of homogeneous agents. By reformulating the problem as a fixed point operator equation, we apply equivariant degree theory to establish rigorous conditions for unbounded global Hopf bifurcation from the consensus equilibrium. Our main results provide sufficient conditions for the local asymptotic stability of consensus and for the existence of unbounded global branches of non-constant periodic solutions with prescribed spatio-temporal symmetries. The question of whether such periodic solutions are stable (and therefore constitute periodic multiconsensus) is not resolved by the degree method; we address it in an illustrative example via numerical simulations. The example, which models eight coupled asset markets with chartist traders and fundamentalists, demonstrates how memory-driven instability can generate periodic boom-bust cycles across clusters of assets. The numerical example illustrates the predicted two-cluster pattern, and Floquet computations provide numerical evidence of local orbital asymptotic stability at sampled parameter values.
\end{abstract}

\section{Introduction}\label{sec:introduction}
Multi-agent systems (MAS) provide models of collective behavior generated by local interactions, including self-organized motion \cite{vicsek1995novel,reynolds1987boid}, opinion formation, synchronization, and cooperative control \cite{li2019survey,amirkhani2022consensus}. Related models of collective behavior include the Hegselmann--Krause bounded-confidence model \cite{hegselmann2002opinion}, the Cucker--Smale flocking model \cite{cucker2007emergent}, and the Kuramoto coupled-oscillator model \cite{kuramoto1984chemical}.
A MAS is a system of multiple independent agents operating in a shared environment based on local information. Agents must coordinate with one another, exchange information, and negotiate bilaterally in order to achieve their goals, in much the same way as human beings do in our own daily interactions. Due to these intricate interactions between agents, MASs can exhibit very rich and complex global behaviors. 

The problem of consensus in MASs has been widely studied. Consensus describes a situation where interactions cause agents to reach an agreement on a common value or state \cite{saber2004consensus,saber2005,jadbabaie2003coordination,ren2005consensus}. As the theory and applications of MASs have expanded to encompass antagonistic or competitive relations between agents, corresponding generalizations of the concept of consensus have also been advanced. These include bipartite consensus \cite{hu2014bipartite, tian2018bipartite} and periodic multiconsensus \cite{chen2014multiconsensus}, where subpopulations of agents display stable oscillatory behaviors. Time delays arise naturally from information transmission and processing. Their effects have been studied for consensus with input and communication delays \cite{tian2008consensus,lin2008average}, delayed multiconsensus \cite{chen2014multiconsensus}, and flocking with distributed memory \cite{choi2019distributed}. In these cases, Hopf bifurcation may be used to show the stability of consensus within certain parameter ranges, and its breakdown into oscillatory behaviors \cite{chen2014multiconsensus,xie2015second}. 

Distributed delays provide a natural way to incorporate continuous memory into consensus protocols. In addition to memory of past states, agents may respond to trends over a finite time interval. Such trend memory can be written as a differentiated history integral, suggesting a connection with neutral equations. For the memory term studied here, however, the fundamental theorem of calculus eliminates the differentiated history integral and gives a retarded equation. We use the term \textit{pseudoneutral} to describe this memory structure and make its retarded character explicit below.

The connection with neutral equations provides a further context for this memory structure. Haskovec \cite{haskovec2022asymptotic} studies consensus with delay and anticipation using neutral delay differential equations, while Wang and Han \cite{wang2024distribution} use the root distribution of quasi-polynomials of neutral type in consensus protocol design. These works motivate the study of feedback involving past rates of change. In the present model, integrating the response over a time window produces a retarded equation, while concentrating that window gives a formal neutral limit, as described below. This connection concerns the construction of the memory term; the compactness and bifurcation results proved here apply to the retarded model.

In this paper, we will define and study the closed-loop dynamics of a MAS which incorporates distributed state memory and feedback from changes over a finite time window, modeled as a system of retarded functional differential equations featuring distributed state memory and pseudoneutral trend memory. We will show conditions for consensus at the trivial solution, and study the problem of Hopf bifurcation of unbounded global branches of solutions classified according to isotropy type. First, let us define the system to be studied and formulate the main results. Consider a system of $n$ identical equations, where the $i$-th equation is given by:
\[
\frac{d}{dt}\left[x_i - \int_0^{\tau_1} g(x_i(t-s))ds\right] = -ax_i - \alpha f\left( \int_0^{\tau_2} x_i(t-s)ds \right) - h_i(x), \;\quad x_i(t) \in \mathbb R.
\]
This system can be viewed as modeling the closed-loop dynamics of a continuous time multi-agent system of $n$ homogeneous agents, where each agent updates its state according to a protocol that incorporates:
\begin{itemize}
\item A self-regulation term $-ax_i(t)$, representing the agent's tendency to return to a reference point at $x\equiv0$.
\item An interaction rule given by some nonlinear coupling function $h_i:\mathbb R^n \to \mathbb R$ which defines a fixed interaction topology with state-dependent effective weights.
\item A nonlinear function of its own continuous state history, given as a function of the distributed retarded delay term $-\alpha f(\int_0^{\tau_2}x_i(t-s)ds)$, where $\alpha$ is taken as a bifurcation parameter.
\item A trend-memory term $\frac{d}{dt}\left[\int_0^{\tau_1}g(x_i(t-s))ds\right]$, representing the change in the accumulated response over a moving time window.
\end{itemize}
where $f,g:\mathbb R\to \mathbb R$ are nonlinear response functions.
Define $\bm f,\bm g,\bm h:\mathbb R^n\to\mathbb R^n$ by
\begin{align*}
\bm f(\bm x)&:=\big(f(x_1),\dots,f(x_n)\big)^T,\\
\bm g(\bm x)&:=\big(g(x_1),\dots,g(x_n)\big)^T,\\
\bm h(\bm x)&:=\big(h_1(\bm x),\dots,h_n(\bm x)\big)^T.
\end{align*}
Then the closed-loop dynamics of the full system can be written
\begin{equation}\label{eq:basic-sys}
\frac{d}{dt}\left[\bm x - \int_0^{\tau_1} \bm g(\bm x(t-s))ds\right] = -a \bm x - \alpha \bm f\left( \int_0^{\tau_2} \bm x(t-s)ds \right) - \bm h(\bm x).
\end{equation}
For continuous $\bm g$ and a continuous trajectory, the fundamental theorem of calculus gives
\begin{equation}\label{eq:pseudoneutral-reduction}
\frac{d}{dt}\int_0^{\tau_1}\bm g(\bm x(t-s))\,ds
=\bm g(\bm x(t))-\bm g(\bm x(t-\tau_1)).
\end{equation}
Thus \eqref{eq:basic-sys} is equivalent to the retarded equation
\begin{equation}\label{eq:explicit-retarded}
\dot{\bm x}(t)=\bm g(\bm x(t))-\bm g(\bm x(t-\tau_1))
-a\bm x(t)-\alpha\bm f\left(\int_0^{\tau_2}\bm x(t-s)\,ds\right)-\bm h(\bm x(t)).
\end{equation}
The trend term records the change in $g(x_i)$ over the preceding interval. If $g$ and the trajectory are continuously differentiable, it also has the exact representation
\[
g(x_i(t))-g(x_i(t-\tau_1))
=\int_{t-\tau_1}^{t}g'(x_i(s))\dot x_i(s)\,ds.
\]
For linear $g(x)=\gamma x$, this is $\gamma$ times the accumulated past rate of change. In a single-integrator interpretation $\dot x_i=v_i$, where $v_i$ is the control input, it can therefore represent feedback from past control actions. For an increasing response $g$, this term reinforces recent changes and provides a possible mechanism for decision persistence.

We call a memory structure \textit{pseudoneutral} when it admits an implicit presentation
\[
\frac{d}{dt}\big[x(t)-\mathcal K(x_t)\big]=\mathcal R(x_t),
\]
where $\mathcal R$ is a continuous functional of the state history alone, and there is a continuous history functional $\mathcal Q$ such that
\[
\frac{d}{dt}\mathcal K(x_t)=\mathcal Q(x_t).
\]
Here $x_t(s)=x(t+s)$, $s\in[-r,0]$, for a finite memory length $r>0$, and the identity holds on continuous trajectories. The equation therefore reduces to the retarded equation $\dot x(t)=\mathcal R(x_t)+\mathcal Q(x_t)$. In this usage, pseudoneutral equations form a subclass of retarded equations, distinguished by their implicit memory structure and their connections with neutral equations. The reduction in \eqref{eq:pseudoneutral-reduction} requires no derivative of $\bm g$.

A further connection with neutral equations is obtained by concentrating the memory kernel. For $\tau_0>0$, $c\in\mathbb R$, and $\varepsilon>0$, put
\[
\mathcal K_\varepsilon(x_t)=\frac{c}{\varepsilon}\int_{\tau_0}^{\tau_0+\varepsilon}x(t-s)\,ds.
\]
Then, on each fixed continuously differentiable trajectory,
\[
\frac{d}{dt}\mathcal K_\varepsilon(x_t)
=\frac{c}{\varepsilon}\big[x(t-\tau_0)-x(t-\tau_0-\varepsilon)\big]
\longrightarrow c\dot x(t-\tau_0).
\]
The factor $1/\varepsilon$ makes the endpoint difference a secant approximation to the delayed derivative. This gives a formal neutral limit on smooth trajectories. Finite-difference approximations of neutral equations have an established numerical literature \cite{kazakova1981}; retarded approximations are also studied by Shampine \cite{shampine2008}. Here the finite memory window is part of the model, and the limit illustrates its relation to neutral feedback; convergence of solutions or bifurcations would require a separate analysis.

The MAS interpretation motivates the feedback and consensus questions, while the conclusions concern the differential equation itself. They apply to any symmetric coupled system of the form \eqref{eq:basic-sys} satisfying the stated hypotheses, independently of the interpretation of its variables. Hopf bifurcation also describes the onset of oscillations in climate models \cite{ghil1983climate}, pedestrian flow \cite{corradi2012pedestrian}, and nerve impulse models \cite{labouriau1989degenerate}. The range of local spectral configurations in applications is illustrated by singular Hopf bifurcation \cite{guckenheimer2012singular}, the Gavrilov--Guckenheimer interaction in a controlled chemical reactor \cite{serra1999cstr}, and bifurcations associated with a three-fold zero eigenvalue \cite{yu1988threefold}.

For systems with memory, related equivariant Hopf methods have been applied to neutral equations modeling transmission-line networks \cite{balanov2007transmission} and to retarded equations with distributed delays in control models \cite{crane2026}. These applications provide a broader context for the use of symmetry and global continuation in delay equations. Here we study retarded equations with pseudoneutral memory and formulate the resulting existence and symmetry criteria in the language of consensus protocols.

To account for the symmetries of the agent interactions, let $\Gamma_0\leq S_n$ be a finite symmetry group which acts by permuting the indices of $\bm x$. Let $V:= \mathbb R^n$, and put $\Gamma := \Gamma_0 \times \mathbb Z_2$, where $\mathbb Z_2$ acts antipodally, i.e. the action of $(\sigma,\pm1)$ on $\bm x \in V$ is given by
\[
    (\sigma,\pm1)\bm x :=  (\sigma,\pm1)(x_1,\dots,x_n)^T=\pm(x_{\sigma(1)},\dots,x_{\sigma(n)})^T.
\]
We will make the following assumptions about $\bm f, \bm g,$ and $\bm h$:
\begin{enumerate}[label=($C_\arabic*$)]\setcounter{enumi}{0}\itemc
\item\label{c1} $\bm f$, $\bm g$, and $\bm h$ are $\Gamma$-equivariant (and therefore are odd functions).
\item\label{c2} $\bm f$, $\bm g$, and $\bm h$ are continuous functions differentiable at 0, and $b := f'(0), \gamma := g'(0), C:= Dh(0)$ satisfy:
    \begin{enumerate}
        \item $b\neq0$.
        \item $0 < \gamma < 1$.
        \item $C$ is a $\Gamma$-equivariant symmetric matrix.
    \end{enumerate}
\item\label{c3} $\bm f$, $\bm g$, and $\bm h$ are locally
Lipschitz in a neighborhood of $x\equiv 0$.
\end{enumerate}

To use a single spatial index throughout, we impose the following \emph{scalar-action assumption} in addition to \ref{c1}--\ref{c3}:
\begin{enumerate}[label=($C_\arabic*$)]\setcounter{enumi}{3}\itemc
    \item\label{c4} $C$ acts as a scalar on each real $\Gamma$-isotypic component $V_j$, i.e.
    \[
    C|_{V_j}=\mu_j\operatorname{Id}_{V_j}.
    \]
\end{enumerate}

Put $a_j:=a+\mu_j$.
\begin{remark}\normalfont
Assumption \ref{c4} simplifies the notation. For a general symmetric equivariant $C$, the spectral theorem gives the invariant decomposition
\[
V_j=\bigoplus_{\mu\in\sigma(C|_{V_j})}\big[V_j\cap\ker(C-\mu\operatorname{Id})\big].
\]
Each summand consists of copies of the same irreducible representation $\mathcal V_j$, and $C$ acts on it as the scalar $\mu$. The arguments extend by assigning a characteristic factor to each such eigenvalue, requiring $a+\mu>0$ for all eigenvalues of $C$ in the stability and continuation results, and summing the contributions with their irreducible multiplicities. The formulas below use \ref{c4} to avoid an additional eigenvalue index. Scalar action on $V_j$ is equivalent to the same action on its complexification.
\end{remark}

We will show that the trivial solution at $\bm x\equiv0$ is locally asymptotically stable, and for this reason we will refer to it as the \textit{reference consensus} (or \textit{consensus equilibrium}). We will also show Hopf bifurcation of global branches of non-constant periodic solutions from the $\bm x\equiv 0$ consensus manifold, give sufficient conditions for these branches to be unbounded, classify these branches according to their spatio-temporal symmetries and, for an illustrative example, use numerical simulations to examine the predicted oscillatory patterns. Specifically, for a system of the form \eqref{eq:basic-sys} satisfying \ref{c1} -- \ref{c4}, we will prove the following theorems:

\begin{theorem}\label{thm:mas-asymp}
Let $\alpha_*$ be the first positive parameter admitting a center, defined by $\alpha_*:=\inf\{\alpha>0:\text{\eqref{eq:system-sincos} holds for some $\beta>0$ and $j$}\}$, with $\inf\emptyset=+\infty$. Suppose additionally that $\bm f$ and $\bm h$ are locally Lipschitz near zero. If $0<\gamma\tau_1\leq1$, $b>0$, and $a_j>0$ for all $j$, then $\bm x\equiv0$ is a locally asymptotically stable consensus equilibrium for all $\alpha\in[0,\alpha_*)$.
\end{theorem}
\begin{theorem}\label{thm:mas-local}
Suppose $b>0$, $a_j>0$ for every $j$, $0<\gamma\tau_1\leq1$, and $0<\tau_1\leq\tau_2$. At every positive center $(\alpha_0,\beta_0)$ satisfying \eqref{eq:system-sincos}, non-constant periodic solutions bifurcate from the reference consensus with limit frequency $\beta_0$.
\end{theorem}
\begin{theorem}\label{thm:mas-global}
Under the hypotheses of Theorem \ref{thm:mas-local}, each such center belongs to an unbounded connected component of the closure of the non-constant periodic solutions in the half-period-antisymmetric space. Here unboundedness is in $(\alpha,\log\beta,u)$, with $u\in\mathscr E$, as specified in Definition \ref{def:branch-global}.
\end{theorem}
\begin{theorem}\label{thm:mas-sym}
Under the hypotheses of Theorem \ref{thm:mas-local}, suppose $D_j(\alpha_0,i\beta_0)=0$ and $(H)\in\mathfrak M_j$. Then there is an unbounded connected component through $(\alpha_0,\beta_0,0)$ in the closure of the non-constant periodic solutions with symmetries at least $H$.
\end{theorem}

Hopf bifurcation provides a framework for studying the onset of oscillations in delayed consensus protocols \cite{chen2014multiconsensus,xie2015second}. Symmetry can force several characteristic roots to cross simultaneously and produce periodic solutions with different spatio-temporal patterns. We use the equivariant twisted Leray--Schauder degree to detect these patterns and establish global continuation. Proposition \ref{prop:pseudoneutral-compactness} supplies the compactness needed for this approach.

Topological degree connects spectral crossings to local bifurcation \cite{krasnoselskii1956} and global continuation \cite{rabinowitz1971,rabinowitz1985}. For periodic solutions, the $S^1$-equivariant degree incorporates time translation \cite{dylawerski1991s1,ize1989equivariant,ize1992degree}. The twisted equivariant degree also records spatial symmetries \cite{geba1994equivariant,balanov2006applied,krawcewicz1997theory,wieslawbook}. In the present setting, the periodic field is a compact perturbation of identity, and its local invariants are computed from characteristic-root crossings on the spatial isotypic components and temporal Fourier modes. Related methods have been applied to distributed-delay systems \cite{crane2026,duan2024}; the present model adds pseudoneutral trend feedback to the distributed state memory.

One notable compromise that equivariant degree techniques must make in order to give such rich symmetric bifurcation results is a lack of information on the stability of solutions, as well as their minimal period. A periodic solution of the closed loop dynamics can only be called multiconsensus if it is at least locally attractive \cite{chen2014multiconsensus}. For the reference consensus, we address this shortcoming by obtaining sufficient conditions on the system parameters to show its local asymptotic stability. For the global branches of non-constant periodic solutions, we advocate a technique of symmetrically informed numerical analysis. By leveraging the information given to us by the degree, in particular the symmetries, limit frequencies, and critical parameter values of bifurcating periodic solutions, we initialize numerical simulations with a perturbation contained in a suitable fixed point subspace corresponding to the detected symmetry of the branch. These simulations reveal the temporal pattern, while the Floquet calculations in Section \ref{sec:floquet} additionally test stability against perturbations outside the selected invariant subspace.

The paper is organized as follows. Section \ref{sec:introduction:overview} introduces the implicit formulation and symmetry notation. Section \ref{sec:crit} treats the characteristic roots, crossing direction, and stability of the equilibrium. Section \ref{sec:funcspace} establishes the compact periodic-field formulation. Sections \ref{sec:two-parameter} and \ref{sec:fixed-point-reduction} compute the local invariants and prove the bifurcation and continuation results. Section \ref{sec:example} presents the coupled-market example and its numerical formulation, followed by the discussion in Section \ref{sec:conclusion}.
\subsection{Preliminaries}\label{sec:introduction:overview}
To apply the equivariant twisted Leray--Schauder degree, we first express the periodic problem as a compact perturbation of identity. The implicit formulation below streamlines the definitions of the periodic operators. We then introduce the representations and twisted orbit types used to describe the spatio-temporal symmetries of solutions. Compactness for the present memory term is proved in Proposition \ref{prop:pseudoneutral-compactness}.

\subsubsection{Implicit formulation and the periodic fixed point field}
Take the system 
\begin{equation}\label{eq:fde2}
\frac d{dt}\big[x(t)- \bm k (x_t)\big]=\bm f(\alpha,x_t), \quad x(t)\in \br^n,
\end{equation}
and, for a fixed number $r>0$, consider the Banach space $\bm C_r:=C([-r,0];\br^n)$ equipped with the standard sup-norm, assume that $\bm f(\alpha,x):\bbR \times \bm C_r \to \br^n$ is a continuous map and $\bm k:\bm C_r\to \br^n$ is continuous, and take $\alpha$ as the bifurcation parameter. This is a functional differential equation written in neutral form; its structural classification depends on the memory functional $\bm k$. In particular, the pseudoneutral case considered here admits a retarded reduction.

We assume that $\bm f$ and $\bm k$ are differentiable at zero, that $D_x\bm f(\alpha,0)$ depends continuously on $\alpha$, and that $\bm k(0)=0$ and $\bm f(\alpha,0)=0$. For the following formulation we require the induced periodic memory map to take $\mathscr E$ continuously into itself. This property and the required compactness are proved for our model in Proposition \ref{prop:pseudoneutral-compactness}; they are not consequences of the implicit form alone.

We wish to reformulate \eqref{eq:fde2} as a fixed point equation on a functional space. We first perform period normalization, using the substitution $x(t) = u(\beta t)$ where $\beta:=\frac{2\pi}{p}$ and $p>0$ is a known period of $x(t)$, to obtain
\begin{equation}\label{eq:fde-norm}
    \frac d{dt}\big[u(t) - \bm k(u_{t,\beta})\big] = \frac{1}{\beta}\bm f(\alpha,u_{t,\beta}), \;\text{ and } u_{t,\beta}(s) := u(t+\beta s)
\end{equation}
Put $\mathscr{E} := C^1_{2\pi}(\bbR;\bbR^n)$ and $||\phi|| =\max\{||\phi||_\infty,||\dot{\phi}||_\infty\}$,
where $||\phi||_\infty$ denotes the usual supremum norm on $\phi$. Let $\bm j: \mathscr{E} \to C_{2\pi}(\bbR;\bbR^n)$ denote the natural embedding of $\mathscr{E}$ into $C_{2\pi}(\bbR;\bbR^n)$. Let $\mathbb R_+^2:=\mathbb R\times(0,\infty)$, and define the periodic memory map by
\[
N_{\bm k}(\alpha,\beta,u)(t):=\bm k(u_{t,\beta}),\qquad u\in\mathscr E.
\]
Under the regularity assumption just stated, its values lie in $\mathscr E$.

Notice that $C_{2\pi}^1(\bbR;\bbR^n) \mathop{\hookrightarrow}\limits^{\bm j}C_{2\pi}(\bbR;\bbR^n)$ is a compact embedding, and put 
\[
N_{\bm f}(\alpha,\beta,u)(t) := \frac{1}{\beta} \bm f(\alpha,u_{t,\beta}).
\]
Define $L:\mathscr E \to C_{2\pi}(\bbR;\bbR^n)$ as $(Lu)(t) := \dot u(t)$. Then the equation \eqref{eq:fde-norm} can be written as
\begin{equation}\label{eq:fde-func1}
L(u-N_{\bm k}(\alpha,\beta,u)) = N_{\bm f}(\alpha,\beta,u),
\end{equation}
where $N_{\bm f}(\alpha,\beta,u)\in C_{2\pi}(\mathbb R;\mathbb R^n)$. Since $L+\bm j:\mathscr E\to C_{2\pi}(\mathbb R;\mathbb R^n)$ is an isomorphism, the periodic problem is equivalent to $\mathscr F(\alpha,\beta,u)=0$, where
\begin{equation}\label{eq:fde-fp-op-eq}
\mathscr F(\alpha,\beta,u)=u-N_{\bm k}(\alpha,\beta,u)
-(L+\bm j)^{-1}\big[N_{\bm f}(\alpha,\beta,u)+\bm j u-\bm jN_{\bm k}(\alpha,\beta,u)\big].
\end{equation}
The applicability of the degree must be checked for this field on its chosen function space. For \eqref{eq:basic-sys}, Proposition \ref{prop:pseudoneutral-compactness} establishes that it is a compact perturbation of identity on compact parameter sets with $\beta>0$.

\subsubsection{Symmetry groups, irreducible representations, and $G$-isotypic decomposition of $\mathscr E$}\label{sec:fourier}
Now we will introduce some standard preliminaries regarding the structure of irreducible $G$-representations, $G$-isotypic components, and how $G$ acts on the functional space $\mathscr E$. Let $\Gamma_0$ be a finite group representing the coupling symmetries of $\eqref{eq:fde2}$, and put $\Gamma:=\Gamma_0\times\bbZ_2$. Suppose $V:=\bbR^n$ is an orthogonal $\Gamma$-representation. If $\bm f,\bm k$ are $\Gamma$-equivariant, then $\mathscr F$ is $G$-equivariant, where $G := S^1 \times \Gamma$, and $S^1$ acts on the period-normalized equation by phase shifting on Fourier modes.

V has a $\Gamma$-isotypic decomposition given by
\[
V = V_0 \oplus V_1 \oplus \dots \oplus V_r
\]
where $V_i$ is the real $\Gamma$-isotypic component modeled on the irreducible $\Gamma$-representation $\cV_i^-$, which is the irreducible $\Gamma_0$-representation $\cV_i$ with the antipodal $\bbZ_2$-action. If $\cV_i$ is an irreducible $\Gamma_0$-representation of \emph{real type}, then
\[
\cV_i^C := \mathbb C \otimes_\bbR \cV_i
\]
is an irreducible complex $\Gamma_0$-representation, and is called the \emph{complexification} of $\cV_i$. On the other hand, if $\cV_i$ is of \emph{complex type}, then it is also an irreducible complex $\Gamma_0$-representation. Therefore
\[
V^C = V_0^C \oplus V_1^C \oplus \dots \oplus V_r^C,
\]
where, for $\cV_i$ of real type
\[
U_i := \cV_i^C
\]
is a complex $\Gamma_0$-isotypic component modeled on $\cV_i^C$, and if $\cV_i$ is of complex type, then
\[
V_i^C = U_i \oplus U_i'
\]
where $U_i$ is a complex $\Gamma_0$-isotypic component modeled on $\cV_i$ and $U_i'$ is a complex $\Gamma_0$-isotypic component modeled on $\overline{\cV_i}$. For simplicity, in this paper we will assume all the irreducible $\Gamma_0$-representations $\cV_i$ are of real type. 

$S^1$ has a two-dimensional real irreducible representation $\mathcal W_k$ for each $k \in \mathbb N$, where 
\[
\mathcal W_k := \{p\cos(kt) + q\sin(kt): p,q\in \mathbb R\}
\]
and $\theta \in S^1$ acts on $u \in \mathcal W_k$ as $\theta u(t) = u(t+\theta)$. Then all irreducible $G$-representation are given by
\[
\mathcal W_{k,j} := \mathcal W_k \otimes \mathcal V_j^-
\]
where $(\theta,\sigma,\pm1) \in G$ acts on $u = (u_1(t),\dots,u_n(t))^T \in \mathcal W_{k,j}$ as
\[
(\theta,\sigma,\pm1)(u_1(t),\dots,u_n(t))^T =\pm (u_{\sigma(1)}(t+\theta),\dots,u_{\sigma(n)}(t+\theta))^T,
\]
where $\sigma \in \Gamma_0$ is viewed as a permutation of the indices of vectors in $\mathbb R^n$. Since every $u \in \mathscr E$ admits a Fourier decomposition, the $G$-isotypic component $W_{k,j}$ (modeled on the irreducible representation $\mathcal W_{k,j})$ is equivalent as a G-space to
\[
\mathscr E_{k,j} := \{c\cos(kt)+d\sin(kt):c,d \in V_j\}
\]
and so the $G$-isotypic decomposition of $\mathscr E$ is given by
\[
\mathscr E := \overline{\bigoplus_{k=0}^\infty \bigoplus_{j=0}^r \mathscr E_{k,j}}
\]

\subsubsection{Orbit types}\label{sec:prelim-orbittypes}
We denote by $\Phi(G)$ the set of all conjugacy classes of subgroups of $G$. A conjugacy class of subgroups $(H)$ is called an \textit{orbit type} if $H=G_u:=\{g\in G:gu=u\}$ for some $u\in\mathscr E\setminus\{0\}$. We denote by $\Phi(G;\mathscr E)$ the set of all orbit types of $G$ on $\mathscr E$. These sets admit a natural partial order given by subconjugacy: $(H) \leq (K)$ if and only if there exists some $H' \in (H)$ and $K'\in (K)$ such that $H' \leq K'$. 

An orbit type is maximal in a specified representation if it is maximal under subconjugacy among the isotropy types of its nonzero vectors. All maximality statements below refer to the indicated Fourier representation. A solution has symmetries at least $H$ when it belongs to the fixed point space $\mathscr E^H$.

The equivariant twisted degree classifies solutions by their spatio-temporal symmetries. In particular, it can detect non-constant periodic solutions whose isotropy $H \leq G$ satisfies $\dim W_G(H) = 1$, where $W_G(H) = N_G(H)/H$ is the Weyl group of $H$ in $G$ and $N_G(H)$ is the normalizer of $H$ in $G$.

These are so-called \textit{twisted orbit types}, and take the form
\[
({\mathbb Z_{km}}^{\mathbb Z_k} \times {}^{H_0}H) := \{(z,h)\in \mathbb Z_{km} \times H: \theta(h)=z^k\}
\]
where $\theta:H \to S^1$ has image $\mathbb Z_m$ and kernel $H_0$. We denote the set of twisted orbit types by $\Phi_1^t(G;\mathscr E)$. 

We will denote the set of maximal twisted orbit types in $\Phi_1^t(G;\mathscr E_{k,j})$ as $\mathfrak M_{k,j}$, and the maximal twisted orbit types in $\Phi_1^t(G;\mathscr E_{1,j})$ simply as $\mathfrak M_j$. For a detailed construction of $\Phi_1^t(G)$ and the twisted degree, see \cite{wieslawbook}.

We use the amalgamated notation ${H_1}^{H_1'}\times{}^{H_2'}H_2$ for a subgroup of $H_1\times H_2$ obtained by matching the images of two epimorphisms onto a common quotient, with kernels $H_1'$ and $H_2'$. The quotient identification is understood as part of this notation. For the spatial subgroups occurring in the example, we use the conventions \cite{golubitsky2012}
\begin{align*}
H^p&:=H\times\mathbb Z_2,&
D_n^z&:={D_n}^{\mathbb Z_n}\times{}^{\mathbb Z_1}\mathbb Z_2,\\
D_{2n}^d&:={D_{2n}}^{D_n}\times{}^{\mathbb Z_1}\mathbb Z_2,&
S_4^-&:={S_4}^{A_4}\times{}^{\mathbb Z_1}\mathbb Z_2,\\
\mathbb Z_{2n}^-&:={\mathbb Z_{2n}}^{\mathbb Z_n}\times{}^{\mathbb Z_1}\mathbb Z_2.
\end{align*}

We also note that maximal twisted orbit types in $G$ may be very easily computed in GAP for any finite $\Gamma_0$ and any representation. Concrete examples of such computations will be illustrated in Section \ref{sec:example}. 

Using the implicit presentation of \eqref{eq:basic-sys}, we obtain the periodic field described in Section \ref{sec:introduction:overview}. Its compactness will be established directly in Section \ref{sec:funcspace}. To apply the equivariant Krasnosel'skii and Rabinowitz theorems, we must compute the local bifurcation invariants and verify the hypotheses of the corresponding continuation result.

This means showing that critical values of $\alpha$ are isolated, and studying the spectrum of our linearized system. We will do this in two stages. First, we will linearize the system \eqref{eq:basic-sys} and study its characteristic operator equation. Then we will perform the same type of reformulation as in Section \ref{sec:introduction:overview}, linearize the period-normalized operator equation, and describe its spectrum in terms of the characteristic operator equation we obtain here. This will provide us with all necessary information to compute local bifurcation invariants and thus ultimately to prove the main results.

\section{Linearization, spectral analysis, and critical points}\label{sec:crit}

Linearizing \eqref{eq:basic-sys} around $x=0$ we obtain
\begin{equation}\label{eq:basic-sys-lin}
\frac{d}{dt} \left[\bm x - \gamma \int_0^{\tau_1} \bm x(t-s)ds\right] + a \bm x + \alpha b \int_0^{\tau_2} \bm x(t-s)ds + C \bm x = 0
\end{equation}

The characteristic operator equation for \eqref{eq:basic-sys-lin} is given by
\begin{equation}\label{eq:char-op}
\triangle_{\alpha}(\lambda):=\left(\lambda + \gamma(e^{-\lambda\tau_1}-1) + a - \alpha b \frac{e^{-\lambda\tau_2} -1}{\lambda}\right)\text{Id} + C
\end{equation}
\begin{definition}[Centers and limit frequencies]\normalfont
The $\bm x \equiv 0$ solution to \eqref{eq:basic-sys-lin} at $\alpha=\alpha_0$ is called a \emph{center} if there exists a corresponding $\beta_0>0$ (called the \textit{limit frequency}) such that $\det_{\mathbb C}\triangle_{\alpha_0}(i\beta_0) =0$. If, in addition, there exists $\varepsilon>0$ such that $0<|\alpha-\alpha_0|+|\beta-\beta_0|<\varepsilon$ implies
\[
{\det}_{\mathbb C}\triangle_{\alpha}(ik\beta) \not=0 
\]
for all $k\in\mathbb N$, then it is called an \emph{isolated center}.
\end{definition}

On the complexified isotypic component $V_j^{\mathbb C}$, define the entire scalar characteristic factor
\begin{equation}\label{eq:scalar-factor}
D_j(\alpha,\lambda)=\lambda+a_j+\gamma(e^{-\lambda\tau_1}-1)
+\alpha b\int_0^{\tau_2}e^{-\lambda s}\,ds.
\end{equation}
Then
\begin{equation}\label{eq:char-op-iso}
\triangle_{\alpha,j}(\lambda)=D_j(\alpha,\lambda)\operatorname{Id}_{V_j^{\mathbb C}},
\end{equation}
and
\begin{equation}\label{eq:char-block-det}
\det_{\mathbb C}\triangle_\alpha(\lambda)=\prod_{j=0}^rD_j(\alpha,\lambda)^{\dim V_j}.
\end{equation}
With the real-type convention in the representation-theoretic preliminaries, the number of copies of $\mathcal V_j$ in $V_j$ is
\begin{equation}\label{eq:def-iso-mult}
m_j=\frac{\dim V_j}{\dim\mathcal V_j}.
\end{equation}
The full spectral determinant counts dimension; the equivariant crossing number counts irreducible multiplicity.
\subsection{The critical set}
For nonzero characteristic roots it is convenient to use
\begin{equation}\label{eq:char-root}
P_j(\alpha,\lambda):=\lambda D_j(\alpha,\lambda)
=\lambda^2+\gamma\lambda(e^{-\lambda\tau_1}-1)+a_j\lambda
-\alpha b(e^{-\lambda\tau_2}-1).
\end{equation}
The zero $P_j(\alpha,0)=0$ is introduced by multiplication by $\lambda$ and need not be a characteristic root. For $\beta>0$, the exact imaginary-root equations are
\begin{equation}\label{eq:system-sincos}
\begin{aligned}
-\beta^2+\gamma\beta\sin(\beta\tau_1)+\alpha b(1-\cos(\beta\tau_2))&=0,\\
\beta\big[a_j+\gamma(\cos(\beta\tau_1)-1)\big]+\alpha b\sin(\beta\tau_2)&=0.
\end{aligned}
\end{equation}
We define the full critical set by
\begin{equation}\label{eq:crit-set}
\Lambda:=\{(\alpha,\beta,0):\alpha>0,\ \beta>0,\
\text{\eqref{eq:system-sincos} holds for some }j\}.
\end{equation}

For an exact scalar frequency equation in the stability regime, assume
\begin{equation}\label{eq:spectral-regime}
b>0,\qquad a_j>0\ \text{for all }j,\qquad 0<\gamma\tau_1\leq1.
\end{equation}
Then $\beta-\gamma\sin(\beta\tau_1)>0$ for every $\beta>0$, and \eqref{eq:system-sincos} excludes $\beta\tau_2=2m\pi$. Put $B_j(\beta)=a_j+\gamma(\cos(\beta\tau_1)-1)$. The full critical equations are equivalent to
\begin{equation}\label{eq:beta-relation}
H_j(\beta):=\big[\beta-\gamma\sin(\beta\tau_1)\big]\cos(\beta\tau_2/2)
+B_j(\beta)\sin(\beta\tau_2/2)=0
\end{equation}
and
\begin{equation}\label{eq:alpha-safe}
\alpha=\frac{\beta\big[\beta-\gamma\sin(\beta\tau_1)\big]}
{b\big[1-\cos(\beta\tau_2)\big]}.
\end{equation}
At a positive zero of $H_j$, $\sin(\beta\tau_2/2)\ne0$, so this last denominator is nonzero.

The following proposition establishes isolation and the finiteness needed in the continuation argument.
\begin{proposition}\label{prop:critical-local-finite}
Under \eqref{eq:spectral-regime}, each spatial isotypic component has infinitely many positive critical frequencies, with no finite accumulation. Only finitely many physical centers have $0<\alpha\leq\alpha_{\max}$ for every finite $\alpha_{\max}>0$. For the period-normalized formulation of Section \ref{sec:funcspace}, the critical set is locally finite on parameter sets with bounded $\alpha\geq0$ and $\beta$ bounded away from zero.
\end{proposition}
\begin{proof}
The function $H_j$ is real analytic, with
\[
H_j'(0)=1-\gamma\tau_1+\tfrac12a_j\tau_2>0.
\]
It is not identically zero, and its positive zeros are isolated and bounded away from zero. At $\beta=2m\pi/\tau_2$, $m\geq1$, its sign is $(-1)^m$, since $\beta-\gamma\sin(\beta\tau_1)>0$. Thus every interval between successive such frequencies contains a zero, and \eqref{eq:alpha-safe} gives a positive critical parameter. If $\alpha\leq\alpha_{\max}$, the first equation in \eqref{eq:system-sincos} implies
\[
\beta^2\leq\gamma\beta+2b\alpha_{\max},\qquad
\beta\leq \beta_{\max}:=\frac{\gamma+\sqrt{\gamma^2+8b\alpha_{\max}}}2.
\]
Analyticity and the exclusion of a neighborhood of zero now give finiteness. In the Fourier decomposition of Section \ref{sec:fourier}, the $k$-th mode of a period-normalized function has physical frequency $k\beta$. On $\beta\geq\beta_{\min}>0$, it satisfies $k\leq \beta_{\max}/\beta_{\min}$, so only finitely many harmonic copies occur there.
\end{proof}

The positive frequencies may therefore be indexed in increasing order as $\beta_{n,j}$, with corresponding parameters $\alpha_{n,j}$ given by \eqref{eq:alpha-safe}. Their ordering is by frequency; the corresponding parameters need not be monotone.

\subsection{Motion of the characteristic roots}
Suppose $P_j(\alpha_0,i\beta)=0$, $\beta>0$, and $\partial_\lambda P_j(\alpha_0,i\beta)\ne0$. The implicit function theorem gives a local root $\lambda(\alpha)=u(\alpha)+iv(\alpha)$ satisfying $P_j(\alpha,\lambda(\alpha))=0$, with $u(\alpha_0)=0$ and $v(\alpha_0)=\beta$. Differentiating this identity with respect to $\alpha$ gives
\[
\partial_\alpha P_j(\alpha,\lambda(\alpha))+
\partial_\lambda P_j(\alpha,\lambda(\alpha))\big(u'(\alpha)+iv'(\alpha)\big)=0.
\]
The partial derivatives, with the other argument held fixed, are
\begin{align*}
\partial_\alpha P_j&=-b(e^{-\lambda\tau_2}-1),\\
\partial_\lambda P_j&=2\lambda+\gamma(e^{-\lambda\tau_1}-1)
-\gamma\tau_1\lambda e^{-\lambda\tau_1}+a_j+\alpha b\tau_2e^{-\lambda\tau_2}.
\end{align*}
Write $\partial_\lambda P_j(\alpha_0,i\beta)=p-iq$, where
\begin{align*}
p&=B_j(\beta)-\gamma\beta\tau_1\sin(\beta\tau_1)
+\alpha_0b\tau_2\cos(\beta\tau_2),\\
q&=\gamma\sin(\beta\tau_1)+\gamma\beta\tau_1\cos(\beta\tau_1)
+\alpha_0b\tau_2\sin(\beta\tau_2)-2\beta.
\end{align*}
Equating real and imaginary parts at $(\alpha_0,i\beta)$ gives
\begin{equation}\label{eq:u'v'}
pu'+qv'=b(\cos(\beta\tau_2)-1),\qquad
-qu'+pv'=-b\sin(\beta\tau_2).
\end{equation}
Consequently,
\begin{equation}\label{eq:pq}
u'(\alpha_0)=\frac{\rho_j(\alpha_0,\beta)}{p^2+q^2},\qquad
\rho_j:=b\big[p(\cos(\beta\tau_2)-1)+q\sin(\beta\tau_2)\big].
\end{equation}
The sign of $u'(\alpha_0)$ determines the direction in which the root crosses the imaginary axis.

Using both equations in \eqref{eq:system-sincos}, the numerator simplifies at every center to
\begin{equation}\label{eq:crossing-numerator}
\begin{aligned}
\rho_j(\alpha_0,\beta)&=b\beta\,\Xi(\beta),\\
\Xi(\beta)&=\tau_2\beta-\sin(\beta\tau_2)+\gamma\big[(\tau_1-\tau_2)\sin(\beta\tau_1)
+\tau_1\sin(\beta(\tau_2-\tau_1))\big].
\end{aligned}
\end{equation}
Indeed, $\beta B_j=-\alpha_0b\sin(\beta\tau_2)$ and $\alpha_0b(1-\cos(\beta\tau_2))=\beta(\beta-\gamma\sin(\beta\tau_1))$ eliminate $B_j$ and $\alpha_0$ using the two imaginary-root identities. The resulting expression is smooth for all $\beta$ and independent of $a_j$, except through the locations of the centers.

\begin{proposition}\label{prop:crossing-num}
Suppose $b>0$, $\gamma>0$, $0<\tau_1\leq\tau_2$, and $\gamma\tau_1\leq1$. Every positive-frequency scalar center is simple and satisfies $u'(\alpha_0)>0$.
\end{proposition}
\begin{proof}
Set $d=\tau_2-\tau_1\geq0$ and $h(z)=z-\sin z$ for $z\geq0$. The function $h$ is strictly increasing, with $h(0)=0$. Equation \eqref{eq:crossing-numerator} becomes
\[
\Xi(\beta)=h(\beta\tau_2)+\gamma d\,h(\beta\tau_1)
-\gamma\tau_1h(\beta d).
\]
Since $\beta\tau_2>\beta d$, $h(\beta d)\geq0$, and $\gamma\tau_1\leq1$, we have
\[
\Xi(\beta)\geq h(\beta\tau_2)-h(\beta d)>0.
\]
Thus $\rho_j=b\beta\Xi(\beta)>0$. If $p=q=0$, the definition of $\rho_j$ would give zero, a contradiction. Hence the scalar root is simple and \eqref{eq:pq} gives $u'(\alpha_0)>0$.
\end{proof}

For equal delays, the formula reduces to $\rho_j=b\beta\big[\tau\beta-\sin(\tau\beta)\big]>0$. More generally, even outside the ordering condition, the bound
\[
\Xi(\beta)\geq\tau_2\beta-1-\gamma(|\tau_1-\tau_2|+\tau_1)
\]
shows that $u'(\alpha_0)>0$ at every center whose frequency satisfies
\[
\beta>\frac{1+\gamma(|\tau_1-\tau_2|+\tau_1)}{\tau_2},
\]
when $b>0$ and $\gamma>0$.

We also wish to find conditions on the system and on the bifurcation parameters which guarantee local asymptotic stability of the trivial solution (i.e. consensus). This requires us to analyze \eqref{eq:char-op} for a given value of $\alpha$, and show that all roots have negative real part. Due to the infinitude of roots of \eqref{eq:char-op}, this can be difficult for arbitrary values of $\alpha$. 

We first prove stability at $\alpha=0$, and then continue it up to the first positive parameter admitting an imaginary characteristic root. Define this parameter by
\[
\alpha_*:=\inf\{\alpha>0:\text{\eqref{eq:system-sincos} holds for some $\beta>0$ and $j$}\},
\]
with $\inf\emptyset=+\infty$. Under \eqref{eq:spectral-regime}, Proposition \ref{prop:critical-local-finite} and the proof below imply that this is the smallest positive center parameter.

\begin{proposition}
Consider \eqref{eq:basic-sys} under \ref{c1}--\ref{c4}, with $\bm f$ and $\bm h$ locally Lipschitz near zero. Let $\alpha_*$ be the first center parameter defined above. If $0<\gamma\tau_1\leq1$, $b>0$, and $a_j>0$ for all $j$, then the trivial solution is locally asymptotically stable for $\alpha\in[0,\alpha_*)$. There is no local steady-state bifurcation from this equilibrium for $\alpha\geq0$.
\end{proposition}

\begin{proof}
At $\alpha=0$, a characteristic root $\lambda=u+iv$ in the closed right half-plane would satisfy
\[
u+\gamma(e^{-u\tau_1}\cos(v\tau_1)-1)+a_j=0,
\qquad v=\gamma e^{-u\tau_1}\sin(v\tau_1).
\]
If $v\ne0$, then
\[
|v|=\gamma e^{-u\tau_1}|\sin(v\tau_1)|
<\gamma\tau_1e^{-u\tau_1}|v|\leq|v|,
\]
a contradiction. If $v=0$, put $r(u)=u+\gamma e^{-u\tau_1}-\gamma$. For $u\geq0$, $r(0)=0$ and $r'(u)=1-\gamma\tau_1e^{-u\tau_1}\geq0$, so $r(u)+a_j>0$. Thus all characteristic roots have negative real part at $\alpha=0$.

For continuation in $\alpha$, use the entire characteristic factor
\[
D_j(\alpha,\lambda)=\lambda+a_j+\gamma(e^{-\lambda\tau_1}-1)
+\alpha b\int_0^{\tau_2}e^{-\lambda s}\,ds.
\]
In particular, $D_j(\alpha,0)=a_j+\alpha b\tau_2>0$ for $\alpha\geq0$. The identity $P_j(\alpha,\lambda)=\lambda D_j(\alpha,\lambda)$ shows that the zero of $P_j$ at $\lambda=0$ is extraneous. Moreover, for $0\leq\alpha\leq\alpha_{\max}$ and $\operatorname{Re}\lambda\geq0$, a root satisfies
\[
|\lambda|\leq a_j+2\gamma+\alpha_{\max}b\tau_2.
\]
Hence roots cannot enter the closed right half-plane from infinity on a compact parameter interval. Continuity of characteristic roots, the absence of a zero root, and \eqref{eq:system-sincos} imply that no root reaches the imaginary axis for $0\leq\alpha<\alpha_*$. The same bound and stability at zero also give $\alpha_*>0$.

The linear retarded equation is therefore exponentially stable throughout this interval. On the phase space $C([-r,0];\mathbb R^n)$, $r=\max\{\tau_1,\tau_2\}$, the right-hand side of \eqref{eq:explicit-retarded} is locally Lipschitz near zero and differentiable at zero. Its nonlinear remainder is $o(\|\varphi\|_\infty)$. The linearized stability principle for retarded equations \cite{hale1993} gives local asymptotic stability of the nonlinear equilibrium. Finally, $D_j(\alpha,0)>0$ excludes local steady-state bifurcation from zero for every $\alpha\geq0$.
\end{proof}

Note that this suffices to prove Theorem \ref{thm:mas-asymp}.

\begin{remark}\normalfont
The condition $\gamma\tau_1\leq1$ is sufficient for the preceding stability argument, but is not necessary. For example, if $a_j>2\gamma$ for every $j$, then at $\alpha=0$ a root with $\operatorname{Re}\lambda\geq0$ would satisfy
\[
|\lambda+a_j-\gamma|=\gamma|e^{-\lambda\tau_1}|\leq\gamma,
\]
whereas $|\lambda+a_j-\gamma|\geq a_j-\gamma>\gamma$. Thus stability holds for every $\tau_1>0$ in this regime, and persists for sufficiently small positive $\alpha$. In the MAS interpretation, the proposition provides a parameter range in which the reference consensus is locally attainable, before any consensus-breaking oscillatory instability.
\end{remark}

\section{Functional space reformulation}\label{sec:funcspace}

In this section, we reformulate \eqref{eq:basic-sys} as a fixed point problem in functional spaces, and show that it is a compact perturbation of identity. We linearize this operator and describe the spectrum of this linearization in terms of the spectrum of \eqref{eq:basic-sys-lin} by leveraging the relationship between $G$-isotypic components of $\mathscr E$ and $\Gamma_0$-isotypic components of $V$. (Note that in this section and for the remainder of the paper, we will use $u(t)$ to refer to a period-normalized function, not to be confused with its usage as the real part of an eigenvalue in the previous section.) 

As with \eqref{eq:fde-norm}, we set $u(t):=\bm x(\tfrac{p}{2\pi}t)$ and $\beta = \tfrac{2\pi}{p}$ and substitute into system \eqref{eq:basic-sys} to obtain the period-normalized system
\begin{equation}\label{eq:basic-sys-norm}
\frac{d}{dt}\left[u- \int_0^{\tau_1} \bm g(u(t-\beta s))ds\right] = -\frac{a}{\beta}u - \frac{\alpha}{\beta}\bm f\left(\int_0^{\tau_2} u(t-\beta s)ds\right) - \frac{1}{\beta}\bm h(u) 
\end{equation}

We define the space $\mathscr E$ and the operators $L, \bm j, N_{\bm f}$ and $N_{\bm k}$ as in Section \ref{sec:introduction:overview} and put
\begin{align*}
N_{\bm k}(\alpha,\beta,u) &= \int_0^{\tau_1}\bm g(u(t-\beta s))ds\\
N_{\bm f}(\alpha,\beta,u) &= -\frac{a}{\beta}u - \frac{\alpha}{\beta}\bm f\left(\int_0^{\tau_2} u(t-\beta s)ds\right) - \frac{1}{\beta}\bm h(u) 
\end{align*}

yielding
\begin{equation}\label{eq:op-sys}
\mathscr F(\alpha,\beta,u)=u-N_{\bm k}(\alpha,\beta,u)
-(L+\bm j)^{-1}\big[N_{\bm f}(\alpha,\beta,u)+\bm j u-\bm jN_{\bm k}(\alpha,\beta,u)\big].
\end{equation}
\begin{proposition}\label{prop:pseudoneutral-compactness}
For continuous $\bm f,\bm g,\bm h$, the field \eqref{eq:op-sys} is a compact perturbation of identity on $\mathscr E=C^1_{2\pi}(\mathbb R;\mathbb R^n)$, jointly on compact parameter sets with $\beta>0$. Under \ref{c1}, this field is $G$-equivariant. Moreover, the implicit and explicit retarded formulations give the same periodic field.
\end{proposition}
\begin{proof}
The fundamental theorem of calculus gives
\begin{equation}\label{eq:periodic-memory-derivative}
L\big(N_{\bm k}(\alpha,\beta,u)\big)=\frac{\bm g(u(t))-\bm g(u(t-\beta\tau_1))}{\beta}.
\end{equation}
Let $(u_m)$ be bounded in $\mathscr E$ and let $(\alpha_m,\beta_m)$ lie in a compact parameter set with $\beta_m>0$. Passing to a subsequence, $u_m\to u$ uniformly and $(\alpha_m,\beta_m)\to(\alpha,\beta)$ with $\beta>0$. Uniform continuity on bounded subsets of $\mathbb R^n$ gives uniform convergence of both $N_{\bm k}(\alpha_m,\beta_m,u_m)$ and their derivatives in \eqref{eq:periodic-memory-derivative}. Hence $N_{\bm k}(\alpha_m,\beta_m,u_m)$ converges in $\mathscr E$. The same subsequence argument shows that $N_{\bm f}$ is compact from $\mathscr E$ into $C_{2\pi}$; the embedding $\bm j$ is compact as well. Since $(L+\bm j)^{-1}$ is bounded, \eqref{eq:op-sys} is a compact perturbation of identity.

Finally, suppressing the parameters to write $N_{\bm k}(u) = N_{\bm k}(\alpha,\beta,u)$, we have the exact identity
\begin{equation}\label{eq:same-periodic-field}
\begin{aligned}
\mathscr F(u)&=u-N_{\bm k}(u)-(L+\bm j)^{-1}[N_{\bm f}(u)+\bm j u-\bm jN_{\bm k}(u)]\\
&=u-(L+\bm j)^{-1}[N_{\bm f}(u)+\bm j u+LN_{\bm k}(u)].
\end{aligned}
\end{equation}
The last expression is precisely the periodic field of \eqref{eq:explicit-retarded}. Equivariance follows from \ref{c1} and commutation of the operators with time translations.
\end{proof}

The equivariant twisted Leray--Schauder degree therefore applies on admissible bounded domains with compact parameter closure in $\mathbb R^2_+$ \cite{krawcewicz1997theory,wieslawbook}.

Following this reformulation, the problem of finding periodic solutions to \eqref{eq:basic-sys} is now equivalent to finding $(\alpha,\beta,u) \in \bbR_+^2 \times \mathscr{E}$ such that $\mathscr{F}(\alpha,\beta,u) = 0$. $\mathscr F$ is also differentiable at every $(\alpha,\beta,0) \in \bbR_+^2 \times \mathscr E$. Put $\mathscr A(\alpha,\beta) := D_u \mathscr F(\alpha,\beta,0)$. To simplify notation, we also define a distributed delay operator:
\begin{align*}
K_{\tau}u(t) &:= \int_0^{\tau}u(t-\beta s)ds
\end{align*}
Then we have
\[
\begin{aligned}
\mathscr A(\alpha,\beta)u &= u - \gamma K_{\tau_1}u - (L+\bm j)^{-1}\left(-\frac{\alpha b}{\beta}K_{\tau_2}u - \frac{1}{\beta}Cu - \Big(\frac{a}{\beta} - 1\Big)u- \gamma K_{\tau_1}u\right)  \\
&= u - \gamma K_{\tau_1}u + (L+\bm j)^{-1}\left(\frac{\alpha b}{\beta}K_{\tau_2}u + \frac{1}{\beta}Cu + \Big(\frac{a}{\beta} - 1\Big)u+ \gamma K_{\tau_1}u\right).
\end{aligned}
\]
The $G$-equivariance of $\mathscr F$ implies that $\mathscr A$ is also $G$-equivariant. Accordingly, it preserves the $G$-isotypic components, and we define the restricted maps $\mathscr A_{k,j}(\alpha,\beta) := \mathscr A_{|\mathscr E_{k,j}}(\alpha,\beta):\mathscr E_{k,j} \to \mathscr E_{k,j}$. 

Note that the characteristic operator equation of the period normalized system \eqref{eq:basic-sys-norm} is given by
\[
\frac{1}{\beta}\triangle_{\alpha}(\lambda)
\]
and the characteristic operator equation for $(L+\bm j)$ is given by
\[
(L+\bm j)e^{\lambda t} = (\lambda + 1)e^{\lambda t}
\]
Combining this with the fact that $u(t) = \bm x(\frac{t}{\beta})$ and taking $\lambda = ik\beta$, we have
\begin{align}
\mathscr A_{k,j}(\alpha,\beta) &= \frac{1}{ik\beta + \beta}\triangle_{\alpha,j}(ik\beta)_{|\mathscr E_{k,j}}\quad &k>0\label{eq:char-Ak}\\
\mathscr A_{0,j}(\alpha,\beta) &= \frac{1}{\beta}\big((a+\alpha b\tau_2)\operatorname{Id}+C\big)|_{V_j}\quad &k=0\label{eq:char-A0}
\end{align}
\section{Two-parameter bifurcation}\label{sec:two-parameter}
Here we will define the required $G$-homotopy invariant quantities (namely the crossing number and the local bifurcation invariant) and show how they are computed. Using the equivariant versions of the existence theorem of M.A. Krasnosel'skii and the unbounded continuation argument of the Rabinowitz alternative, we will prove Hopf bifurcation of global branches and conditions for their unboundedness under the all-frequency crossing condition, allowing simultaneous spatial and temporal crossings. 

We can view the system \eqref{eq:basic-sys-norm} as a two-parameter bifurcation problem given by
\[
\mathscr F(\alpha,\beta,u) = 0, \quad (\alpha,\beta,u) \in \bbR^2_+ \times \mathscr E
\]
Indeed, one can easily see that from the functional space reformulation shown above that the following facts hold:
\begin{enumerate}
    \item $\mathscr F$ is a $G$-equivariant compact perturbation of identity, jointly on compact parameter sets.
    \item For all $(\alpha,\beta) \in \bbR^2_+$, $\mathscr F(\alpha,\beta,0)=0$. 
    \item $\mathscr A(\alpha,\beta):=D_u\mathscr F(\alpha,\beta,0)$ depends continuously on the parameters and, locally uniformly in $(\alpha,\beta)$,
    \[
    \mathscr F(\alpha,\beta,u)=\mathscr A(\alpha,\beta)u+o(\|u\|),\qquad u\to0.
    \]
\end{enumerate}

We will now give a few basic definitions related to symmetric Hopf bifurcation in a topological context, which are necessary to describe bifurcating branches and their symmetries. Denote by $\mathscr M$ and $\mathscr S$ the sets of all trivial and nontrivial solutions to \eqref{eq:basic-sys-norm} respectively. Then we have 
\begin{align*}
\mathscr M &:= \{(\alpha,\beta,0) \in \bbR^2_+ \times \mathscr E:\quad \mathscr F(\alpha,\beta,0)=0\}\\
\mathscr S &:= \{(\alpha,\beta,u)\in\bbR^2_+ \times \mathscr E:\quad\mathscr F(\alpha,\beta,u)=0,\quad u\not\equiv0\}
\end{align*}
Clearly this is a partition of the set of solutions of \eqref{eq:basic-sys-norm}.
\begin{definition}\normalfont
A point $(\alpha_0,\beta_0,0)$ is called a \textit{bifurcation point} of \eqref{eq:basic-sys-norm} if for every neighborhood $U\subset \bbR^2_+ \times \mathscr E$ containing $(\alpha_0,\beta_0,0)\in \mathscr M$, we have $U \cap \mathscr S \not= \emptyset$.
\end{definition}
\begin{definition}\label{def:branch-global}\normalfont
A branch is a connected component of $\overline{\mathscr S}$ containing a nontrivial continuum; it bifurcates from a trivial solution $z_0$ if $z_0$ belongs to that component, in which case $z_0$ is called a branching point. For prescribed symmetries $H$, the component is taken in $\overline{\mathscr S\cap(\mathbb R_+^2\times\mathscr E^H)}$. To include both increasing and decreasing frequency in global continuation, we follow \cite[eq.~(71)]{balanov2023parabolic} and use the product of the $\mathscr E$ norm and the parameter metric
\[
d((\alpha,\beta),(\alpha',\beta'))=
\sqrt{(\alpha-\alpha')^2+(\log\beta-\log\beta')^2}.
\]
Continuation is taken in $\mathbb R\times(0,\infty)\times\mathscr E$, although the centers considered here have $\alpha>0$. Thus unboundedness allows $\beta\to0$ (unbounded period), as well as growth of $|\alpha|$, $\beta$, or $\|u\|$. 

\end{definition}

Note that every branching point is also a bifurcation point. From this definition, we obtain a necessary condition for $(\alpha_0,\beta_0,0)\in \mathscr M$ to be a bifurcation point. Namely, if $(\alpha_0,\beta_0,0)$ is a branching point, then $\mathscr A(\alpha_0,\beta_0):\mathscr E \to \mathscr E$ is not an isomorphism. Equations \eqref{eq:char-Ak} and \eqref{eq:char-A0} imply that this requires either $\det_{\mathbb C} \triangle_{\alpha_0}(ik\beta_0) = 0$ for some $k\in {\mathbb N}$, or $\tfrac{\alpha_0 b \tau_2 + a_j}{\beta_0} = 0$. This provides the connection between the critical set for the system \eqref{eq:basic-sys}, given by \eqref{eq:crit-set}, and the critical set of the period normalized system \eqref{eq:basic-sys-norm}. 

However, if $\tfrac{\alpha_0 b \tau_2 + a_j}{\beta_0} = 0$, this corresponds to a steady-state bifurcation. Since we are interested in showing the bifurcation of non-constant periodic solutions, we view this as a degeneracy. Excluding such points leads us to the non-constant critical set for \eqref{eq:basic-sys-norm}, denoted $\tilde\Lambda$: 

\[
\tilde{\Lambda} := \left\{(\alpha,\beta,0) \in \mathscr M:\;\; \exists_{k \in {\mathbb N}} \;\;{\det}_{\mathbb C} \triangle_\alpha(ik\beta) = 0 \text{ and } \det\triangle_\alpha(0)\ne0\right\}.
\]
Note that $(\alpha,\beta,0)\in\tilde{\Lambda}$ implies $(\alpha,k\beta,0)\in\Lambda$ for some $k\in{\mathbb N}$, and so $\Lambda \subseteq \tilde\Lambda$. At an isolated point $(\alpha_0,\beta_0,0) \in \tilde{\Lambda}$, for some sufficiently small $\varepsilon,\delta > 0$, one can find an \textit{isolated $G$-invariant neighborhood} $\Omega(\alpha_0,\beta_0,0)$ given by
\[
\Omega(\alpha_0,\beta_0,0) := \{(\alpha,\beta,u) \in \bbR^2_+ \times \mathscr E:\;\; |(\alpha,\beta)-(\alpha_0,\beta_0)| < \delta,\;\; ||u|| < \varepsilon\}
\]
such that
\[
\{(\alpha,\beta,u): |(\alpha,\beta)-(\alpha_0,\beta_0)|=\delta,\ \|u\|\leq\varepsilon\}\cap\overline{\mathscr S}=\emptyset.
\]

Now, following the equivariant degree construction for compact periodic fields \cite{krawcewicz1997theory,wieslawbook,balanov2023parabolic}, by using an equivariant analog of the Tietze-Dugundji theorem, one can construct a continuous $G$-equivariant auxiliary function $\eta:\bbR^2_+ \times \mathscr E \to \bbR$ satisfying
\[
\begin{cases} \eta(\alpha,\beta,0)<0 &\text{ if } \; |(\alpha,\beta)-(\alpha_0,\beta_0)|  = \delta,\\
\eta(\alpha,\beta,v)>0 &\text{ if } \; |(\alpha,\beta)-(\alpha_0,\beta_0)|\le \delta \text{ and }\; \|v\|=\ve.\\
\end{cases}
\]

Now define $\mathscr F_\eta : \overline{\Omega(\alpha_0,\beta_0,0)} \to \mathbb R \times \mathscr E$ by 
\[
\mathscr F_\eta(\alpha,\beta,u) = (\eta(\alpha,\beta,u),\mathscr F(\alpha,\beta,u)).
\]
The augmentation has a finite-dimensional first component, and Proposition \ref{prop:pseudoneutral-compactness} makes the second component a compact perturbation of identity. Thus $(\mathscr F_\eta,\Omega(\alpha_0,\beta_0,0))$ forms an admissible $G$-pair. The local bifurcation invariant $\omega_G(\alpha_0,\beta_0)\in A_1^t(G)$ can then be defined:
\[
\omega_G(\alpha_0,\beta_0) := G\text{-deg}(\mathscr F_\eta, \Omega(\alpha_0,\beta_0,0)),
\]
where $G$-deg is the twisted $G$-equivariant Leray--Schauder degree. It is clear from the homotopy invariance of the degree that $\omega_G(\alpha_0,\beta_0)$ is independent of the choice of $\varepsilon,\delta,$ and the auxiliary function $\eta$. 

This local bifurcation invariant is the basis of the equivariant version of the local Hopf bifurcation theorem of Krasnosel'skii, which states that if some orbit type $(H)$ has a non-zero coefficient in $\omega_G(\alpha_0,\beta_0)$, then there exists a branch of solutions with symmetries at least $(H)$ and having limit frequency $\beta_0$ bifurcating from $(\alpha_0,\beta_0,0)$. However, for this result to be practically useful in our case, we need a workable computational formula for $\omega_G(\alpha_0,\beta_0)$.

\subsection{Crossing numbers}
A nonzero root velocity is a convenient sufficient crossing condition, but the degree method permits crossings without scalar simplicity. We define crossing numbers through the Brouwer degree of a complex determinant, viewed as a map from $\mathbb R^2$ to $\mathbb R^2$.

\begin{definition}[Crossing numbers]\label{def:cross-num}\normalfont
For a physical critical frequency $\omega>0$ at $\alpha_0$, choose a rectangle
\[
U_\omega=\{\lambda\in\mathbb C:0<\operatorname{Re}\lambda<\varepsilon,
\ |\operatorname{Im}\lambda-\omega|<\delta\}
\]
and $\alpha_-<\alpha_0<\alpha_+$ sufficiently close so that the characteristic determinant is nonzero on $\partial U_\omega$ at $\alpha_\pm$, with no crossings through its horizontal or right edges in between. Taking $\operatorname{deg_B}$ to represent the Brouwer degree, the total crossing number is
\[
\mathfrak t(\alpha_0,\omega)
:=\deg_B(\det_{\mathbb C}\triangle_{\alpha_-},U_\omega,0)
-\deg_B(\det_{\mathbb C}\triangle_{\alpha_+},U_\omega,0).
\]
The complex orientation is used. Since the determinants are holomorphic in $\lambda$, these degrees count roots in $U_\omega$ with algebraic multiplicity. 
\end{definition}

\begin{definition}[Isotypic and resonant crossing numbers]\normalfont
The determinant on the $j$-th $\Gamma$-isotypic component is $d_j(\alpha,\lambda):=D_j(\alpha,\lambda)^{m_j}$. Define
\[
\begin{aligned}
\mathfrak t_j(\alpha_0,\omega)
&:=\deg_B(d_j(\alpha_-,\cdot),U_\omega,0)
-\deg_B(d_j(\alpha_+,\cdot),U_\omega,0),\\
\mathfrak t_{k,j}(\alpha_0,\beta_0)&:=\mathfrak t_j(\alpha_0,k\beta_0).
\end{aligned}
\]
Noncritical blocks contribute zero. The full count weights $\mathfrak t_j$ by $\dim\mathcal V_j$, whereas the equivariant invariant uses $\mathfrak t_j$ itself, counting copies of the irreducible representation.
\end{definition}

\begin{lemma}\label{lem:sign-crossing-num}
If $D_j(\alpha_0,ik\beta_0)=0$ is a simple scalar root with nonzero real velocity $u'(\alpha_0)$, then
\[
\mathfrak t_{k,j}(\alpha_0,\beta_0)=-m_j\operatorname{sign}u'(\alpha_0).
\]
Under Proposition \ref{prop:crossing-num}, this is $-m_j$ on every critical block and zero on every other block.
\end{lemma}
\begin{proof}
A scalar root with $u'>0$ enters the right-half-plane rectangle and contributes $-m_j$ to the before-minus-after count. A root with $u'<0$ exits and contributes $m_j$.
\end{proof}

We will now present three foundational results needed to complete our proofs. All three are proved in the general equivariant case in \cite{krawcewicz1997theory} and \cite{wieslawbook}, and we will simply restate them here in the context of our system for the reader's convenience.

The first is the computational formula for local bifurcation invariants, which depends upon the crossing number and on the twisted basic degrees $\text{\rm deg}_{\cV_{k,j}^-}$, which can be easily computed in GAP:
\begin{theorem}
Given system \eqref{eq:basic-sys} with $\bm f,\bm g,$ and $\bm h$ satisfying \ref{c1} -- \ref{c4}, if $(\alpha_0,\beta_0,0) \in \tilde{\Lambda}$ is an isolated critical point, then
\[
\omega_G(\alpha_0,\beta_0) = \text{\rm $\Gamma$-deg}(\mathscr A_{0}(\alpha_0,\beta_0),B(V))\cdot \sum_{k=1}^\infty\sum_{j=0}^r\mathfrak t_{k,j}(\alpha_0,\beta_0)\text{\rm deg}_{\cV_{k,j}^-}
\]

\end{theorem}
The second is the equivariant extension of the local bifurcation theorem of Krasnosel'skii:
\begin{theorem}\label{thm:localbif}
Given system \eqref{eq:basic-sys} with $\bm f,\bm g,$ and $\bm h$ satisfying \ref{c1} -- \ref{c4}, if $(\alpha_0,\beta_0,0) \in \tilde{\Lambda}$ is an isolated critical point with $\omega_G(\alpha_0,\beta_0)\not=0$, then there exists a branch of nontrivial solutions to $\eqref{eq:basic-sys}$ bifurcating from the trivial branch at $\alpha= \alpha_0$. Moreover, if $(H)\in\Phi_1^t(G)$ is an orbit type with
\[
\text{coeff}^H(\omega_G(\alpha_0,\beta_0)) \not=0
\]
then there exists a branch of solutions to \eqref{eq:basic-sys} with symmetries at least $(H)$ bifurcating from $(\alpha_0,0)$ with limit frequency $\beta_0$.
\end{theorem}
The global continuation principle used below is the following componentwise form of the Rabinowitz alternative: for a compact equivariant field with a locally finite critical set, a bounded component of the closure of nontrivial solutions, contained in the parameter domain and with only isolated trivial intersections, satisfies
\[
\sum_{z\in\mathscr C\cap\tilde\Lambda}\omega(z)=0.
\]
The sum is finite. Equivalently, a nonzero sum forces the component to meet the boundary of any admissible bounded neighborhood containing it. We apply this principle after the fixed point reduction, where nonzero constant solutions are excluded. The same principle applies to the residual $S^1$ action in a prescribed fixed point space \cite{krawcewicz1997theory,balanov2006applied,balanov2023parabolic}.

Constant solutions are, of course, trivially periodic, and Theorem \ref{thm:localbif} does not distinguish between constant and non-constant periodic solutions. It is also possible, in the global case, that a branch of non-constant periodic solutions could collapse to constant solutions after bifurcation. The equivariant degree offers a convenient tool for handling these situations in the form of fixed point reduction, which allows us to consider the problem on a reduced subspace consisting only of solutions with certain prescribed symmetries. In the following section, we use this to enforce half-period antisymmetry, thereby excluding nonzero constant solutions.

\section {Fixed point reduction}\label{sec:fixed-point-reduction}

Fix $\kappa \in {\mathbb N}$ and consider the subgroup of $G$ generated by 
\[
\bm K := \left<\left(\bm e,-1,e^{\tfrac{i\pi}{\kappa}}\right)\right> \leq \Gamma_0 \times \bbZ_2 \times S^1
\]
We then examine \eqref{eq:op-sys} in the fixed point space $\mathscr E^{\bm K}$, i.e. we are interested in solutions to the system
\begin{equation}\label{eq:func-sys-fixedpoint}
\mathscr F^{\bm K}(\alpha,\beta,u) = 0,\quad (\alpha,\beta,u)\in \bbR^2_+ \times \mathscr E^{\bm K}    
\end{equation}
The fixed point condition is $u(t+\pi/\kappa)=-u(t)$, i.e. half-period antisymmetry. It excludes every nonzero constant function. Every solution of \eqref{eq:func-sys-fixedpoint} is also a solution of \eqref{eq:op-sys}, and the isotypic decomposition of $\mathscr E^{\bm K}$ is given by
\[
\mathscr E^{\bm K}_{k,j} := \begin{cases}
    \mathscr E_{k,j} \quad &\text{ if $k=(2l-1)\kappa$ for some $l\in{\mathbb N}$}\\
    0 \quad &\text{ otherwise}
\end{cases} 
\]
which immediately yields
\[
\mathscr E^{\bm K} = \overline{\bigoplus_{l=1}^\infty \bigoplus_{j=0}^r \mathscr E_{(2l-1)\kappa,j}}
\]
Choosing $\kappa$ equal to the largest critical Fourier index isolates that temporal index in the critical linearization. The substitution $u(t)=v(\kappa t)$ identifies $\mathscr E^{\bm K}$ with the half-period-antisymmetric space $v(s+\pi)=-v(s)$ and rescales the frequency to $\kappa\beta$.

Clearly $\bm K$ is normal in $G$, and since $G_0 := G/\bm K = S^1 \times \Gamma_0$, the above system is a two-parameter $S^1 \times \Gamma_0$-equivariant bifurcation problem, and the set of critical points for \eqref{eq:func-sys-fixedpoint} can be described as follows:
\[
\tilde{\Lambda}^{\bm K}:=\{(\alpha_0,\beta_0,0) \in \bbR^2_+ \times \mathscr E: \exists l \in {\mathbb N}\;\;{\det}_{\mathbb C}(\triangle_{\alpha_0}(i(2l - 1)\kappa\beta_0))=0\}
\]
We also define
\[
\tilde{\Lambda}_k^{\bm K}:=\{(\alpha_0,\beta_0,0) \in \bbR^2_+ \times \mathscr E: \exists j\in\{0,\ldots,r\} \;\;{\det}_{\mathbb C}(\triangle_{\alpha_0,j}(i(2k - 1)\kappa\beta_0))=0\}
\]
and so
\[
\tilde\Lambda^{\bm K} = \bigcup_{k=1}^\infty \tilde\Lambda_k^{\bm K}
\]

Under \eqref{eq:spectral-regime}, the critical set is locally finite on bounded parameter sets with $\beta$ bounded away from zero, by Proposition \ref{prop:critical-local-finite}. If also $\tau_1\leq\tau_2$, Proposition \ref{prop:crossing-num} and Lemma \ref{lem:sign-crossing-num} give negative isotypic crossing numbers on every critical block.

The zero Fourier mode is absent, so the reduced local invariant is
\[
\omega^{\bm K}(z)=\sum_{l\geq1}\sum_{j=0}^r
\mathfrak t_{(2l-1)\kappa,j}(z)\deg_{\mathcal V^-_{(2l-1)\kappa,j}}.
\]
Only finitely many terms are nonzero at each center. For a maximal twisted orbit type $(H)$ in a contributing basic representation, its basic-degree coefficient is
\[
\operatorname{coeff}^{H}(\deg_{\mathcal V^-_{k,j}})
=\frac{\dim_{\mathbb R}(\mathcal V^-_{k,j})^H}{2|W(H)/S^1|}>0.
\]
Folding preserves these coefficients. In particular, when $(H)$ is maximal among the contributing orbit types, the nonzero terms in its coefficient have the same sign. This is the maximal-type noncancellation used in the local invariant \cite{balanov2006applied,wieslawbook}; see also \cite[Appendix C, eq.~(153)]{balanov2023parabolic}.

\begin{proof}[Proof of Theorems \ref{thm:mas-local}--\ref{thm:mas-sym}]
Take $\kappa=1$ and put $X=\mathscr E^{\bm K}$. For the symmetry assertion put instead $X=\mathscr E^H$. Since $(H)\in\mathfrak M_j$ is an orbit type on the first Fourier mode, it contains $(\bm e,-1,e^{i\pi})$, so again every nonzero element of $X$ is non-constant. Time translations preserve $X$ and the restriction of $\mathscr F$ is a compact $S^1$-equivariant field.

Let $n_{k,j}(X)$ denote the complex dimension of the subspace of the $k$-th Fourier mode in $X$ with coefficients in $V_j^{\mathbb C}$. In the $S^1$ local invariant, the coefficient corresponding to the $k$-th Fourier mode is, with the crossing convention above,
\[
c_k^X(z)=-\sum_{j:D_j(\alpha,ik\beta)=0}n_{k,j}(X)\leq0.
\]
Indeed, restriction to $X$ retains scalar characteristic factors with these nonnegative multiplicities, and Proposition \ref{prop:crossing-num} gives their common crossing direction. At the stated center $z_0=(\alpha_0,\beta_0,0)$, $c_1^X(z_0)<0$. The local degree theorem therefore supplies a nontrivial continuum through $z_0$ in $X$.

Let $\mathscr C$ be its component in the closure of the nonzero solutions in $\mathbb R_+^2\times X$. Suppose it were bounded in the metric of Definition \ref{def:branch-global}. Then $\alpha$ and $\|u\|$ are bounded and $\beta$ lies in a compact subinterval of $(0,\infty)$. Compactness of the field gives compactness of the solution set there. All positive-frequency centers have $\alpha>0$: the first equation of \eqref{eq:system-sincos} and $\beta-\gamma\sin(\beta\tau_1)>0$ imply this, while $\cos(\beta\tau_2)=1$ is impossible. Proposition \ref{prop:critical-local-finite} therefore makes the trivial intersections finite. The componentwise continuation identity, applied to the residual $S^1$ degree and its coefficient corresponding to the first Fourier mode, yields the contradiction
\[
0=\sum_{z\in\mathscr C\cap\tilde\Lambda^X}c_1^X(z)<0,
\]
where the strict inequality follows from the term at $z_0$. This contradiction proves unboundedness of the component in $X$ and establishes all three assertions.
\end{proof}

Unboundedness refers to the full connected component in Definition \ref{def:branch-global}; it does not specify its amplitude or parameter projection.

\section{Illustrative example: coupled asset markets}\label{sec:example}
Heterogeneous agent models (HAMs) describe financial markets through the interaction of traders using different expectations and trading rules. An early contribution to this approach is the disequilibrium price-adjustment model of Beja and Goldman \cite{beja1980disequilibrium}. A common distinction in this literature is between fundamentalists, who respond to deviations from a fundamental value, and chartists or momentum traders, who respond to price trends or derivatives. Subsequent continuous-time models have incorporated memory explicitly. Dibeh \cite{dibeh2005speculative} studies delayed chartist expectations and the resulting persistent price oscillations, while He and Zheng \cite{he2010moving} incorporate moving-average trading rules through distributed delays and nonlinear, saturating demand functions.

Related models also consider interactions between asset markets. Dobrescu, Neamtu, and Mircea \cite{dobrescu2016asset} investigate bifurcation and synchronization in two coupled chartist-fundamentalist markets with delays, and Dibeh and El Deeb \cite{dibeh2024synchronization} study synchronization and limit cycles in coupled delayed market models. In this section, we develop these modeling ideas in a symmetric network of markets with two distinct memory mechanisms: a fundamentalist response to accumulated mispricing, and a momentum response to the net price change over a finite time horizon, which may be interpreted as response to pricing trends. The latter gives the pseudoneutral structure studied in this paper. Together with saturating responses and cross-market coupling, these mechanisms provide a setting in which to investigate the existence and symmetries of collective periodic price fluctuations.

There are two levels of agency in this interpretation. Within each market, the heterogeneous agents are the traders, whose behavior is represented through aggregate fundamentalist and momentum-trading responses. At the network level considered here, each entire market is itself viewed as an agent, with a price-adjustment protocol whose reference objective is price discovery. Thus identical market protocols are compatible with heterogeneous trader populations within each market. Since the state of each market measures its deviation from its own fundamental value, the consensus equilibrium at zero represents successful price discovery across the network. Periodic multiconsensus instead describes coordinated fluctuations in mispricing, with groups of markets sharing periodic patterns while remaining distinct from other groups. Such coordination can therefore coexist with persistent departures from fundamental values.

\begin{figure}[tbp]
\centering
\begin{minipage}{0.4\textwidth}
\centering
\begin{tikzpicture}[
    scale=2.5,
    x={(1cm,0cm)},          
    y={(0cm,1cm)},          
    z={(-0.5cm,-0.5cm)},    
    vertex/.style={circle, draw=black, fill=black, inner sep=1.5pt},
    edge/.style={solid, thick, black},
    facediag/.style={dashed, thick, black},
    spacediag/.style={dotted, thick, black},
    every label/.style={font=\tiny},  
]
\coordinate (x1) at (0,1,1);   
\coordinate (x2) at (1,1,1);   
\coordinate (x3) at (0,1,0);   
\coordinate (x4) at (1,1,0);   
\coordinate (x5) at (0,0,1);   
\coordinate (x6) at (1,0,1);   
\coordinate (x7) at (0,0,0);   
\coordinate (x8) at (1,0,0);   

\foreach \i in {1,...,8}
    \node[vertex] at (x\i) {};

\node[above=3pt] at (x1) {$x_3$};
\node[above=3pt] at (x2) {$x_4$};
\node[above=3pt] at (x3) {$x_1$};
\node[above=3pt] at (x4) {$x_2$};

\node[below=3pt] at (x5) {$x_7$};
\node[below=3pt] at (x6) {$x_8$};
\node[below=3pt] at (x7) {$x_5$};
\node[below=3pt] at (x8) {$x_6$};

\foreach \i/\j in {1/2, 1/3, 1/5, 2/4, 2/6, 3/4, 3/7, 4/8, 5/6, 5/7, 6/8, 7/8}
    \draw[edge] (x\i) -- (x\j);

\foreach \i/\j in {1/4, 2/3, 5/8, 6/7, 3/8, 4/7, 1/6, 2/5, 1/7, 3/5, 2/8, 4/6}
    \draw[facediag] (x\i) -- (x\j);

\foreach \i/\j in {1/8, 2/7, 3/6, 4/5}
    \draw[spacediag] (x\i) -- (x\j);

\end{tikzpicture}
\end{minipage}%
\begin{minipage}{0.45\textwidth}
\centering
\small
\resizebox{\linewidth}{!}{%
    \(
    C = \begin{pmatrix}
        0 & c_1 & c_1 & c_2 & c_1 & c_2 & c_2 & c_3 \\
        c_1 & 0 & c_2 & c_1 & c_2 & c_1 & c_3 & c_2 \\
        c_1 & c_2 & 0 & c_1 & c_2 & c_3 & c_1 & c_2 \\
        c_2 & c_1 & c_1 & 0 & c_3 & c_2 & c_2 & c_1 \\
        c_1 & c_2 & c_2 & c_3 & 0 & c_1 & c_1 & c_2 \\
        c_2 & c_1 & c_3 & c_2 & c_1 & 0 & c_2 & c_1 \\
        c_2 & c_3 & c_1 & c_2 & c_1 & c_2 & 0 & c_1 \\
        c_3 & c_2 & c_2 & c_1 & c_2 & c_1 & c_1 & 0
    \end{pmatrix}
    \)%
    }
\end{minipage}
\caption{Cube-coupled system with coupling matrix $C$. Solid, dashed, and dotted lines indicate interactions with coupling strength given by $c_1,c_2$, and $c_3$ respectively.}\label{fig:cube}
\end{figure}
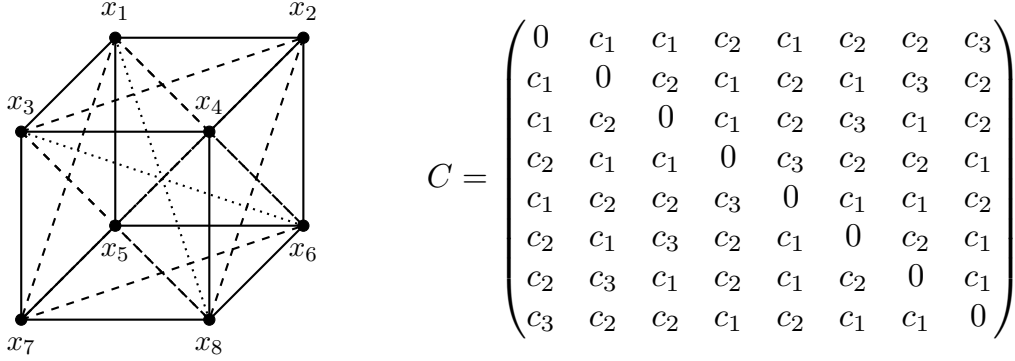
We consider a network of eight coupled homogeneous asset markets. To motivate the cubic coupling configuration seen in Figure \ref{fig:cube}, we suppose that each market is distinguished by three binary characteristics, giving eight possible distinct combinations. We assume that markets which share two characteristics and differ in one are most strongly coupled through the larger economy; markets which share one characteristic have medium coupling; and markets which differ in all three characteristics have the weakest coupling. These coupling strengths are respectively represented by solid, dashed, and dotted lines in Figure \ref{fig:cube}.

Write the state as $x(t)\in\mathbb R^8$, where $x_i(t)$ denotes the difference of the trading price of the $i$-th asset from its underlying fundamental value, which we assume to be fixed. We normalize the fundamental value at 0, so $x_i(t)>0$ indicates that $x_i$ is overvalued at time $t$, and $x_i(t)<0$ indicates that $x_i$ is undervalued at time $t$. We consider each market as being populated by two types of trader: fundamentalist traders, who react to accumulated asset price deviations over a finite time horizon $\tau_2$, and chartist traders, who react to accumulated changes (i.e. trends) in the asset-price response over some finite time horizon $\tau_1$. Then the state dynamics of the $i$-th asset market are given by
\[
\begin{aligned}
\frac{d}{dt}\left[x_i-\gamma\int_0^{\tau_1}x_i(t-s)\,ds\right]
&=-ax_i-\alpha\tanh\left(b\int_0^{\tau_2}x_i(t-s)\,ds\right)\\
&\quad-\tanh\left(\sum_{j=1}^8c_{ij}x_j\right).
\end{aligned}
\]
Here $-ax_i$ represents a liquidity provider reacting instantly to mispricing, $\alpha$ measures the strength of fundamentalist trading, and the final term represents saturating cross-market coupling. In the notation of \eqref{eq:basic-sys}, $g(x)=\gamma x$, $f(x)=\tanh(bx)$, and $h_i(x)=\tanh((Cx)_i)$.

In this case, the distributed delay terms may also be thought of as smoothing out the noise or high-frequency price flunctuations often seen in real markets. Each market can be viewed as a Heterogeneous Agent Model, similar to 

This forms a multi-agent system where the goal of each agent is price discovery. The existence of consensus at $x\equiv0$ has a natural interpretation as successful price discovery, where all 8 assets are priced at their fundamental value. Multiconsensus represents situations where different clusters of assets are undergoing boom and bust cycles, which could also be viewed as the formation and collapse of pricing bubbles. 

To choose parameters, we will suppose that momentum traders look at trends over the past month, $\tau_1 = 20$, and fundamentalists over the past quarter, $\tau_2=60$. We will set $a = 0.5$, $b =0.2$, $\gamma = 0.04$, $c_1 = 0.15$, $c_2 = 0.05$, and $c_3 = 0.01$. This leads to the following linearized system:
\begin{equation}\label{eq:example-sys-lin}
\frac{d}{dt} \left[x - 0.04 \int_0^{20} x(t-s)ds\right] + 0.5x + 0.2\alpha \int_0^{60} x(t-s)ds + Cx = 0
\end{equation}
where
\[
C = \begin{bmatrix}
0 & 0.15 & 0.15 & 0.05 & 0.15 & 0.05 & 0.05 & 0.01 \\
0.15 & 0 & 0.05 & 0.15 & 0.05 & 0.15 & 0.01 & 0.05 \\
0.15 & 0.05 & 0 & 0.15 & 0.05 & 0.01 & 0.15 & 0.05 \\
0.05 & 0.15 & 0.15 & 0 & 0.01 & 0.05 & 0.05 & 0.15 \\
0.15 & 0.05 & 0.05 & 0.01 & 0 & 0.15 & 0.15 & 0.05 \\
0.05 & 0.15 & 0.01 & 0.05 & 0.15 & 0 & 0.05 & 0.15 \\
0.05 & 0.01 & 0.15 & 0.05 & 0.15 & 0.05 & 0 & 0.15 \\
0.01 & 0.05 & 0.05 & 0.15 & 0.05 & 0.15 & 0.15 & 0
\end{bmatrix},
\]
We note that system $\eqref{eq:example-sys-lin}$ satisfies assumptions \ref{c1} -- \ref{c4}. Before taking its isotypic decomposition to apply the main result, we will first make some brief remarks on the group theory of the symmetry group of rigid motions of a cube/octahedron, denoted $\mathbb O$.

First we note that since $C$ can be written as a weighted sum of adjacency matrices of three undirected cubic graphs corresponding to cubic edges, face diagonals, and space diagonals, respectively. Therefore, $C$ is clearly $\mathbb O$-equivariant. To better describe the action of $\mathbb O$ on our 8 markets, we will consider $\mathbb O \cong S_4 \leq S_8$, the full symmetry group of the 8 vertices. $\mathbb O$ has order 24 and can be generated by a rotation of order 4 about an axis connecting the centers of two opposite faces, a rotation of order 3 about an axis connecting two space diagonal opposite vertices, and a rotation of order 2 about an axis between the midpoints of two opposite edges. We note the following correspondence between elements of $\mathbb O$ (written as permutations in $S_4$) and elements in $S_8$:
\begin{table}[tbp]
\centering
\small  
\begin{tabular}{lll}
Rotation type & $S_4$ representative & $S_8$ representative\\
\hline
Face rotation   & (1234)       & (1243)(5687) \\
Vertex rotation & (142)        & (283)(167)   \\
Edge rotation   & (34)         & (18)(27)(34)(56) \\
\end{tabular}
\caption{Generators of $S_4$ with corresponding generators in $S_8$.}
\end{table}
Let $\mathbb O$ act by permuting vertices in $V:=\mathbb R^8$. Then $V$ is an isometric $\mathbb O$-representation with the following character table:
\begin{table}[tbp]
\centering
\small
\begin{tabular}{|c|ccccc|}
\hline
con. classes & $(1)$ & $(12)$ & $(12)(34)$ & $(123)$ & $(1234)$ \\ \hline
$\chi_0$ & $1$ & $1$ & $1$ & $1$ & $1$ \\
$\chi_1$ & $1$ & $-1$ & $1$ & $1$ & $-1$ \\
$\chi_2$ & $2$ & $0$ & $2$ & $-1$ & $0$ \\
$\chi_3$ & $3$ & $-1$ & $-1$ & $0$ & $1$ \\
$\chi_4$ & $3$ & $1$ & $-1$ & $0$ & $-1$ \\ \hline
$\chi_V$ & $8$ & $0$ & $0$ & $2$ & $0$ \\ \hline
\end{tabular}
\caption{Irreducible characters of $S_4$ with character of $V$.}
\label{table:char}
\end{table}
Let $\Gamma = \mathbb O \times \bbZ_2$. Then $V$ is also an isometric $\Gamma$-representation, and from the above character table we immediately obtain the isotypic decomposition
\[
V = V_0 \oplus V_1 \oplus V_3 \oplus V_4
\]
where each $V_i$ has isotypic multiplicity 1 and is modeled on the $\cV_i^-$ irreducible representation, i.e., the $\cV_i$ irreducible $\mathbb O$-representation equipped with the antipodal $\bbZ_2$-action. Let $G = S^1 \times \mathbb O \times \bbZ_2$. Then our system \eqref{eq:example-sys-lin} can be reformulated as a $G$-symmetric two-parameter bifurcation problem as described in Section \ref{sec:funcspace} and our main results can be applied. We compute the eigenspaces of $C$ by looking for $S_4$-invariant subspaces corresponding to the irreducible representations in the isotypic decomposition:
\begin{align*}
    E(\mu_0) = \text{span}(&(1,1,1,1,1,1,1,1)^T) \quad &\mu_0 &= 3c_1 +3c_2 +c_3\\
    \\
    E(\mu_1) = \text{span}(&(1,-1,-1,1,-1,1,1,-1)^T) \quad &\mu_1 &= -3c_1 + 3c_2 - c_3\\
    \\
    E(\mu_3) = \text{span}(&(1,-1,-1,1,1,-1,-1,1)^T, \quad &\mu_3 &= -c_1 - c_2 + c_3\\
    &(1,1,-1,-1,-1,-1,1,1)^T,\\
    &(1,-1,1,-1,-1,1,-1,1)^T)\\
    \\
    E(\mu_4) = \text{span}(&(1,1,1,1,-1,-1,-1,-1)^T,\quad &\mu_4 &= c_1 - c_2 - c_3\\&(1,1,-1,-1,1,1,-1,-1)^T,\\
    &(1,-1,1,-1,1,-1,1,-1)^T)\\
\end{align*}
By computing traces of conjugacy classes of permutations in $\mathbb O$ on these eigenspaces and comparing them to Table \ref{table:char}, one immediately finds $E(\mu_0) = V_0$, $E(\mu_1) = V_1$, $E(\mu_3) = V_3$, and $E(\mu_4) = V_4$. Therefore we have
\begin{align*}
    a_0 &= a + 3c_1 + 3c_2+c_3 &= 0.5 +0.45 + 0.15 + 0.01 &= 1.11\\
    a_1 &= a - 3c_1+3c_2-c_3 &= 0.5 -0.45 + 0.15 - 0.01 &= 0.19\\
    a_3 &= a - c_1-c_2+c_3 &= 0.5 -0.15 - 0.05 + 0.01 &= 0.31\\
    a_4 &= a + c_1-c_2-c_3 &= 0.5 +0.15 - 0.05 - 0.01 &= 0.59
\end{align*}

Now we can compute the critical set on each isotypic component. For the purposes of this example and numerical simulations, we will compute explicit values only for $0<\beta_{n,j} < 1$ and take only the first few values of $\alpha_{n,j}$. 

By numerically solving equations \eqref{eq:system-sincos}, we obtain the following approximate values for $\alpha_{n,j},\beta_{n,j}$, and isotypic crossing number $\mathfrak t_j(\alpha_{n,j},\beta_{n,j})$:
\begin{align*}
V_0:\quad&\alpha_{1,0} = 4.24124372, & \beta_{1,0} = 0.10259439, &\quad\quad\mathfrak t_0(\alpha_{1,0},\beta_{1,0}) = -1\\
&\alpha_{2,0} =2.48213950,  & \beta_{2,0} = 0.20211387,&\quad\quad\mathfrak t_0(\alpha_{2,0},\beta_{2,0}) = -1\\
&\alpha_{3,0} = 3.24298830,  & \beta_{3,0} = 0.30501229,&\quad\quad\mathfrak t_0(\alpha_{3,0},\beta_{3,0}) = -1\\
V_1:\quad&\alpha_{1,1} = 0.09529711, & \beta_{1,1} = 0.09239073, &\quad\quad\mathfrak t_1(\alpha_{1,1},\beta_{1,1}) = -1\\
&\alpha_{2,1} =0.11199127,  & \beta_{2,1} =0.17498149, &\quad\quad\mathfrak t_1(\alpha_{2,1},\beta_{2,1}) = -1\\
&\alpha_{3,1} = 0.28822644,  & \beta_{3,1} = 0.27963367, &\quad\quad\mathfrak t_1(\alpha_{3,1},\beta_{3,1}) = -1\\
V_3:\quad&\alpha_{1,3} = 0.27955407, & \beta_{1,3} = 0.09705997, &\quad\quad\mathfrak t_3(\alpha_{1,3},\beta_{1,3}) = -1\\
&\alpha_{2,3} =0.22102020,  & \beta_{2,3} = 0.18546039, &\quad\quad\mathfrak t_3(\alpha_{2,3},\beta_{2,3}) = -1\\
&\alpha_{3,3} = 0.43829429, & \beta_{3,3} =0.28779672, &\quad\quad\mathfrak t_3(\alpha_{3,3},\beta_{3,3}) = -1\\
V_4:\quad&\alpha_{1,4} = 1.12332838, & \beta_{1,4} = 0.10070033, &\quad\quad\mathfrak t_4(\alpha_{1,4},\beta_{1,4}) = -1\\
&\alpha_{2,4} =0.70441289,  & \beta_{2,4} = 0.19591782, &\quad\quad\mathfrak t_4(\alpha_{2,4},\beta_{2,4}) = -1\\
&\alpha_{3,4} = 1.06011854, & \beta_{3,4} = 0.29795401, &\quad\quad\mathfrak t_4(\alpha_{3,4},\beta_{3,4}) = -1\\
\end{align*}
Here we can see that the first bifurcation point occurs on the $V_1$ isotypic component. Inspecting the corresponding one-dimensional eigenspace of $C$, we can see that this corresponds to a situation where edge-adjacent markets are in anti-phase, and face-diagonal markets are in phase. In other words, if periodic multiconsensus occurs on this isotypic component, then despite a positive correlation between the prices of these assets, one cluster of assets will be experiencing a pricing bubble while the other experiences a crash, and these boom/bust cycles will alternate due to the competing influences of fundamentalists and momentum traders.

\begin{figure}[tbp]
    \centering
    \begin{minipage}{0.45\textwidth}
        \centering
        \includegraphics[width=\linewidth]{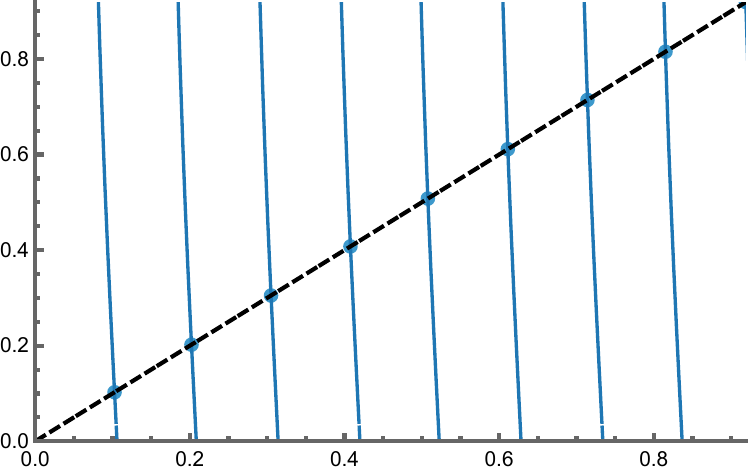}
        \subcaptiontext{a}{Limit frequencies $\beta_{n,0}$ on  $V_0$}
    \end{minipage}\hfill
    \begin{minipage}{0.45\textwidth}
        \centering
        \includegraphics[width=\linewidth]{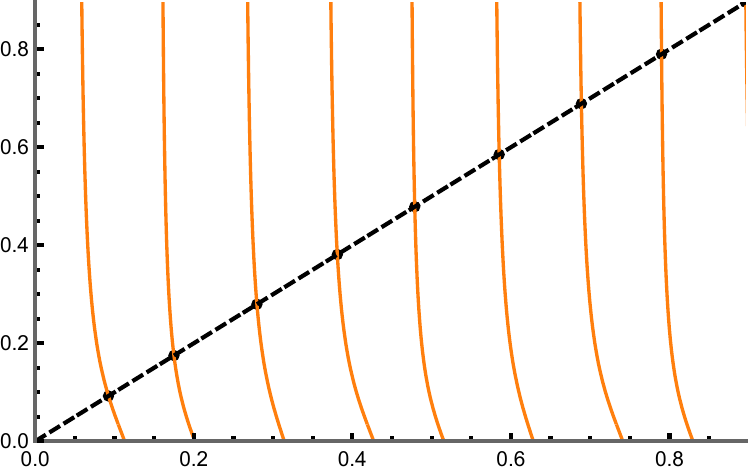}
        \subcaptiontext{b}{Limit frequencies $\beta_{n,1}$ on  $V_1$}
    \end{minipage}
    
    \vspace{0.5cm}
    
    \begin{minipage}{0.45\textwidth}
        \centering
        \includegraphics[width=\linewidth]{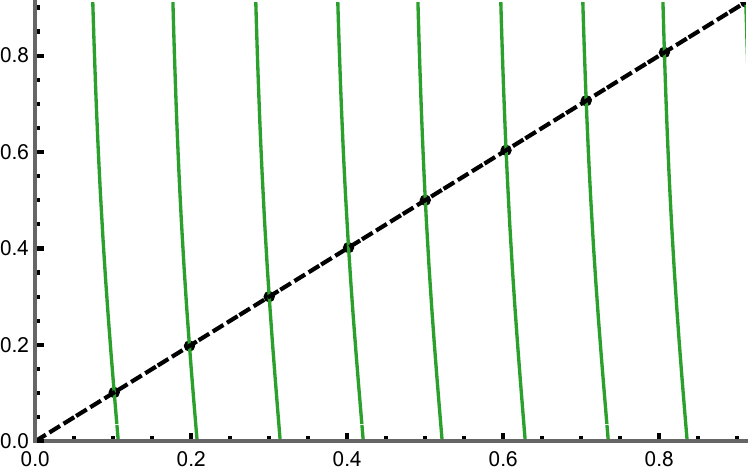}
        \subcaptiontext{c}{Limit frequencies $\beta_{n,3}$ on  $V_3$}
    \end{minipage}\hfill
    \begin{minipage}{0.45\textwidth}
        \centering
        \includegraphics[width=\linewidth]{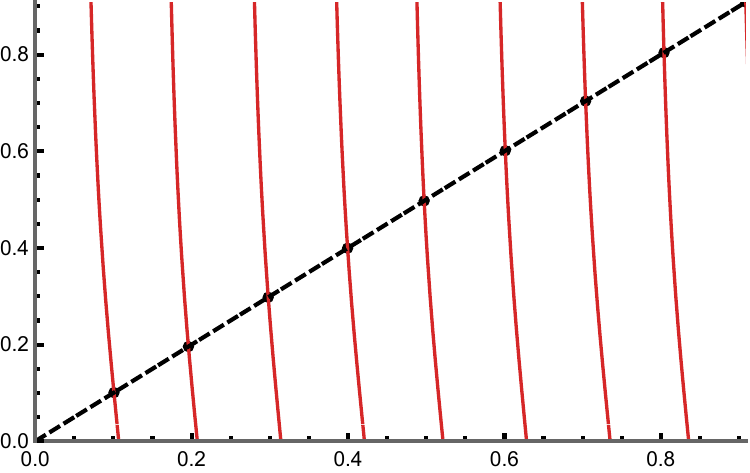}
        \subcaptiontext{d}{Limit frequencies $\beta_{n,4}$ on  $V_4$}
    \end{minipage}
    \caption{Coincidence plots of \eqref{eq:system-sincos} across isotypic components}
    \label{fig:beta-coinc}
\end{figure}
\begin{figure}[tbp]
    \centering
    \begin{minipage}{0.45\textwidth}
        \centering
        \includegraphics[width=\linewidth]{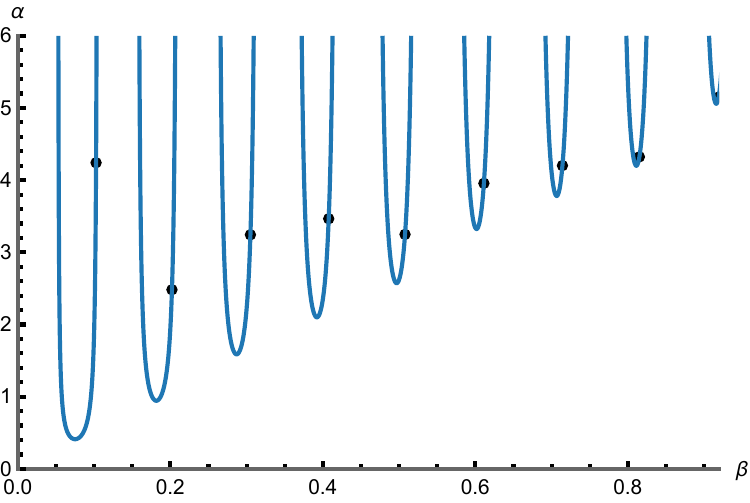}
        \subcaptiontext{a}{Critical points $\alpha_{n,0}$ on  $V_0$}
    \end{minipage}\hfill
    \begin{minipage}{0.45\textwidth}
        \centering
        \includegraphics[width=\linewidth]{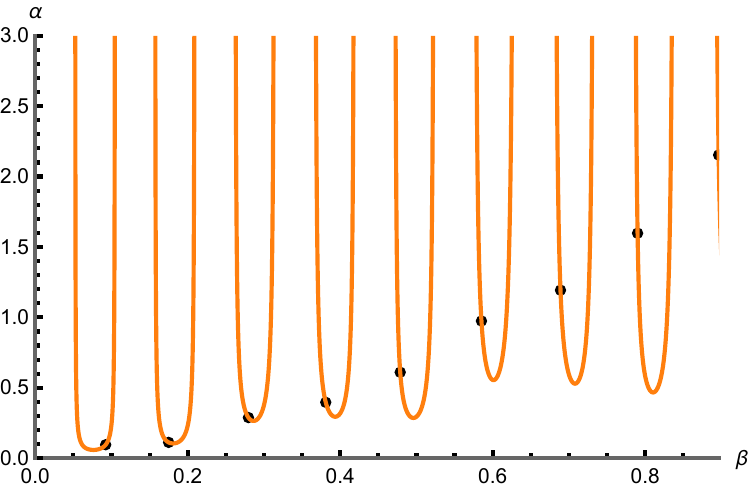}
        \subcaptiontext{b}{Critical points $\alpha_{n,1}$ on  $V_1$}
    \end{minipage}
    
    \vspace{0.5cm}
    
    \begin{minipage}{0.45\textwidth}
        \centering
        \includegraphics[width=\linewidth]{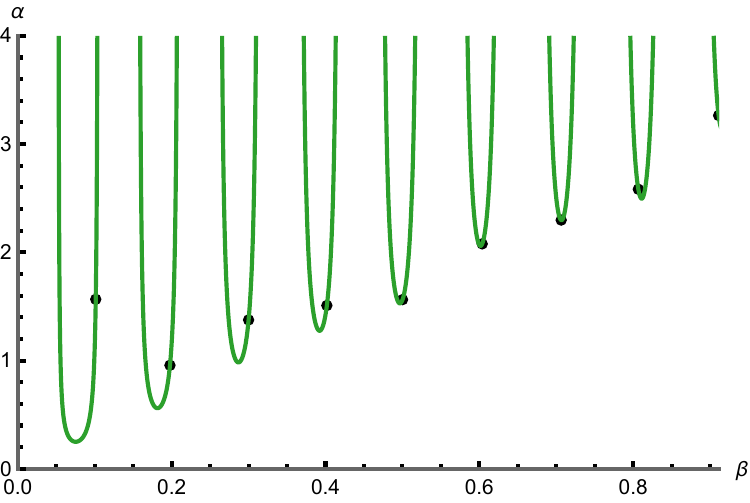}
        \subcaptiontext{c}{Critical points $\alpha_{n,3}$ on  $V_3$}
    \end{minipage}\hfill
    \begin{minipage}{0.45\textwidth}
        \centering
        \includegraphics[width=\linewidth]{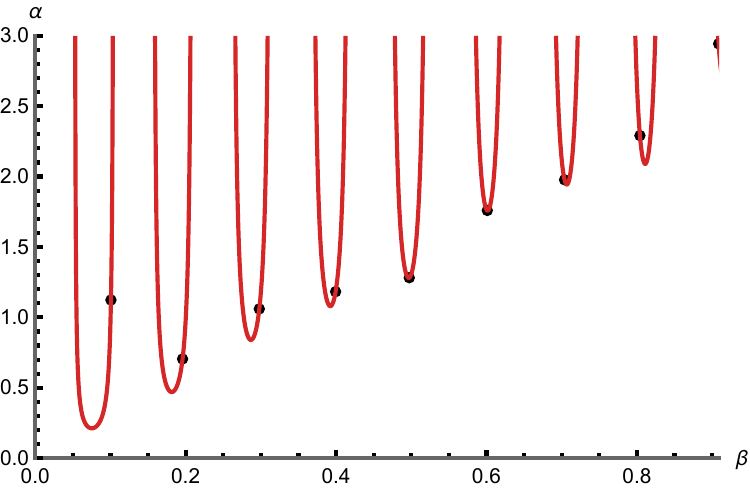}
        \subcaptiontext{d}{Critical points $\alpha_{n,4}$ on  $V_4$}
    \end{minipage}
    \caption{Critical points $\alpha_{n,j}$ corresponding to limit frequencies $\beta_{n,j}$ on each isotypic component, with $0<\alpha_{n,j}<1$}
    \label{fig:alpha-coinc}
\end{figure}

For each $V_j$, the corresponding irreducible $G$-representations are given by $\mathcal V_{m,j}^-$. Using GAP with the Equideg package, we can compute the twisted basic degrees $\deg_{\mathcal V_{m,j}^-}$, yielding
\begin{align*}
    \text{deg}_{\cV_{m,0}^-} = &( \bbZ_{2m} {}^{\bbZ_m}\times {}^{S_4} S_4^p)\\
    \text{deg}_{\cV_{m,1}^-} = &( \bbZ_{2m} {}^{\bbZ_m}\times {}^{S_4^-} S_4^p)\\
    \text{deg}_{\cV_{m,3}^-} = &(\bbZ_{2m} {}^{\bbZ_m}\times {}^{D_4^z} D_4^p) + (\bbZ_{2m} {}^{\bbZ_m}\times {}^{D_3^z} D_3^p) + (\bbZ_{2m} {}^{\bbZ_m}\times {}^{D_2^d} D_2^p)\\ 
    +&(\bbZ_{4m} {}^{\bbZ_m}\times {}^{\bbZ_2^-} \bbZ_4^p)+(\bbZ_{6m} {}^{\bbZ_m}\times \bbZ_3^p) - (\bbZ_{2m} {}^{\bbZ_m}\times {}^{\bbZ_2^-} D_2^p) - (\bbZ_{2m} {}^{\bbZ_m}\times {}^{D_1^z} D_1^p)\\
    \text{deg}_{\cV_{m,4}^-} =&(\bbZ_{2m} {}^{\bbZ_m}\times {}^{D_4^d} D_4^p) + (\bbZ_{2m} {}^{\bbZ_m}\times {}^{D_3} D_3^p) + (\bbZ_{2m} {}^{\bbZ_m}\times {}^{D_2^d} D_2^p)\\ 
    +&(\bbZ_{4m} {}^{\bbZ_m}\times {}^{\bbZ_2^-} \bbZ_4^p)(\bbZ_{6m} {}^{\bbZ_m}\times \bbZ_3^p) - (\bbZ_{2m} {}^{\bbZ_m}\times {}^{\bbZ_2^-} D_2^p) - (\bbZ_{2m} {}^{\bbZ_m}\times {}^{D_1} D_1^p)
\end{align*}
On the $m$-th Fourier mode, the maximal orbit types are listed below; setting $m=1$ gives $\mathfrak M_j=\mathfrak M_{1,j}$.
\begin{align*}
    \mathfrak M_{m,0} = \{&( \bbZ_{2m} {}^{\bbZ_m}\times {}^{S_4} S_4^p)\}\\
    \mathfrak M_{m,1} = \{&( \bbZ_{2m} {}^{\bbZ_m}\times {}^{S_4^-} S_4^p)\}\\
    \mathfrak M_{m,3} = \{&(\bbZ_{2m} {}^{\bbZ_m}\times {}^{D_4^z} D_4^p),(\bbZ_{2m} {}^{\bbZ_m}\times {}^{D_3^z} D_3^p),(\bbZ_{2m} {}^{\bbZ_m}\times {}^{D_2^d} D_2^p),(\bbZ_{4m} {}^{\bbZ_m}\times {}^{\bbZ_2^-} \bbZ_4^p),\\ 
    &(\bbZ_{6m} {}^{\bbZ_m}\times \bbZ_3^p)\}\\
    \mathfrak M_{m,4} =\{&(\bbZ_{2m} {}^{\bbZ_m}\times {}^{D_4^d} D_4^p),(\bbZ_{2m} {}^{\bbZ_m}\times {}^{D_3} D_3^p),(\bbZ_{2m} {}^{\bbZ_m}\times {}^{D_2^d} D_2^p),(\bbZ_{4m} {}^{\bbZ_m}\times {}^{\bbZ_2^-} \bbZ_4^p),\\ 
    &(\bbZ_{6m} {}^{\bbZ_m}\times \bbZ_3^p)\}
\end{align*}
Hence, by Theorems \ref{thm:mas-local}, \ref{thm:mas-global}, and \ref{thm:mas-sym}, if a bifurcation point occurs on the component $V_j$, then for each $(H)\in \mathfrak M_j$ there is an unbounded component of non-constant periodic solutions with symmetries at least $H$, in the sense of Definition \ref{def:branch-global}. Since, in this example, the first bifurcation point occurs on the $V_1$ component, there must be a branch of non-constant periodic solutions having symmetries at least $( \bbZ_{2} {}^{\bbZ_1}\times {}^{S_4^-} S_4^p)$. To determine if this solution represents multiconsensus, we will numerically simulate the system with a perturbation initialized on the $V_1$ subspace for values of $\alpha$ near $\alpha_{1,1} = 0.09529711$. 
\begin{figure}[h]
    \centering
    \includegraphics[width=\linewidth]{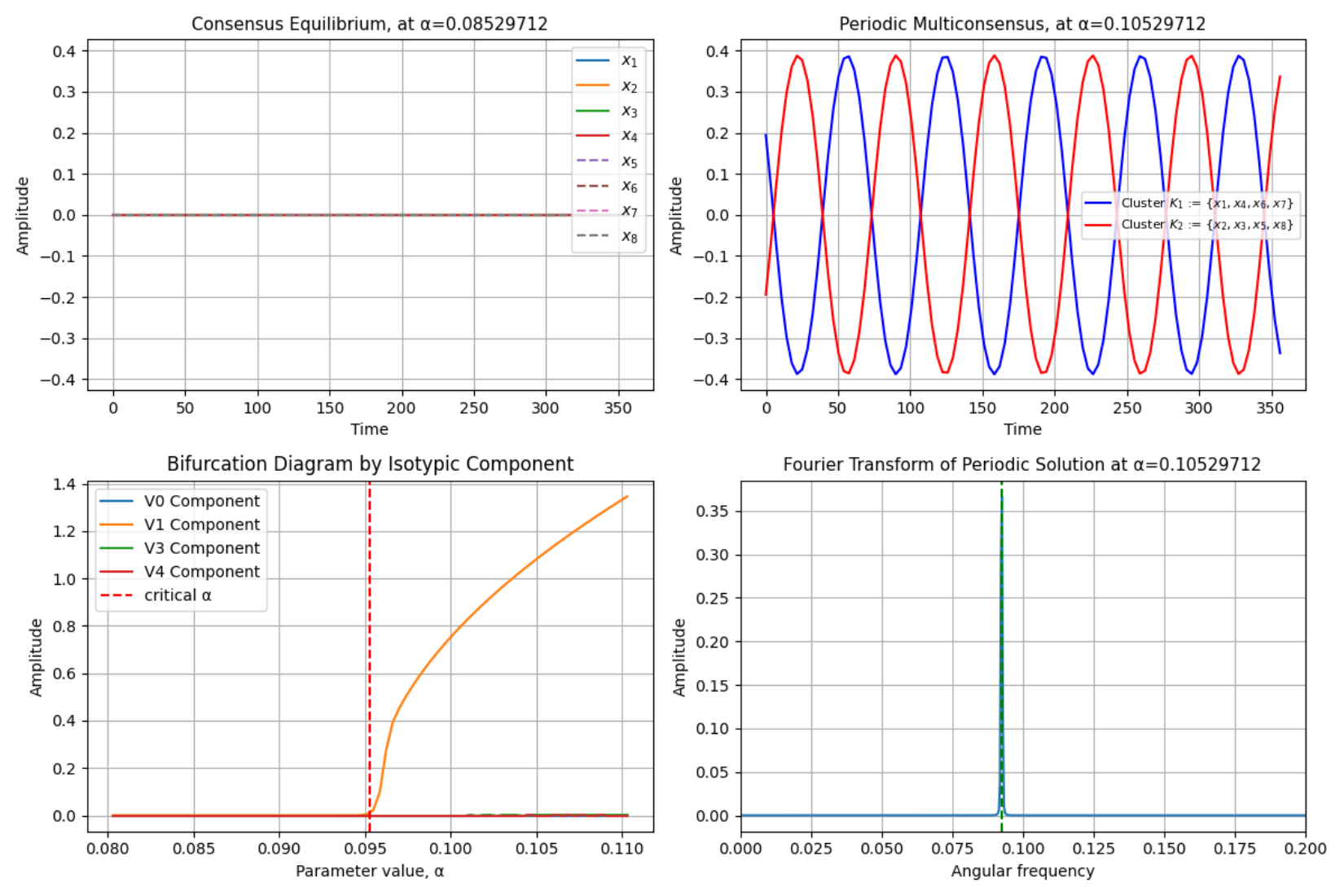}
    \caption{Time series of solution (after discarding initial transient) initialized by a perturbation near $x\equiv0$ on either side of $\alpha_{1,1}$. The dashed red line in the bottom left figure represents the value of $\alpha_{1,1}$, where bifurcation is predicted to occur, and the dashed green line in the bottom right plot represents the value of the limit frequency $\beta_{1,1}$ compared with the fast Fourier transform of the time series data.}\label{fig:bif-diagram}
\end{figure}
The simulations illustrate decay toward consensus below the first center and the predicted two-cluster oscillations above it. Runs initialized with random spatial perturbations also approach this pattern. Section \ref{sec:floquet} examines its local stability through Floquet multipliers in all four spatial components. The frequency $\beta_{1,1}\approx0.09239073$ corresponds to the limiting period $2\pi/\beta_{1,1}\approx68.01$ trading days.

\subsection{Numerical formulation}
The simulations use the JiTCDDE \cite{ansmann2018} library and are implemented in Python. We use an auxiliary-variable formulation analogous to the linear chain trick \cite{hurtado2019linear}, allowing us to avoid expensive quadrature of the distributed delay integral; for the uniform kernel, this gives a system with discrete delays \cite{cassidy2019equivalences}. Introduce
\[
I_i(t):=\int_{t-\tau_2}^t x_i(s)\,ds,\qquad
\dot I_i(t)=x_i(t)-x_i(t-\tau_2).
\]
Together with \eqref{eq:explicit-retarded}, this gives
\[
\dot x_i(t)=\gamma[x_i(t)-x_i(t-\tau_1)]-ax_i(t)
-\alpha\tanh(bI_i(t))-\tanh((Cx(t))_i).
\]
The initial integral must satisfy $I_i(0)=\int_{-\tau_2}^0x_i(s)\,ds$. With this compatibility condition the augmented system reproduces the distributed-delay equation. Simulations initialized along $V_1$ isolate the two-cluster pattern predicted by the symmetry calculation; random spatial initial perturbations also test attraction from outside that subspace.

Figure \ref{fig:bif-diagram} shows representative trajectories initialized using the same seeded random spatial direction in $\mathbb R^8$, at $\alpha=\alpha_*\pm0.01$, where $\alpha_*=\alpha_{1,1}$ is the first critical parameter value. For the bifurcation diagram, we sweep $\alpha$ across an evenly spaced mesh of 80 values in $[\alpha_*-0.015,\alpha_*+0.015]$. Representative parameter values are evaluated additionally when they do not belong to this mesh.

We use absolute and relative integration tolerances $10^{-10}$ and maximum step $0.25$. Trajectories are integrated to $t=90000$, and the final $1800$ time units are recorded at spacing $0.125$. This output spacing is independent of the adaptive integration steps chosen by JiTCDDE. Below the first center, an earlier stopping rule is used when the sampled state histories and integral histories scaled by $\tau_2$ remain below $10^{-8}$ in magnitude at two consecutive checks. A complete $1800$-unit terminal window is then recorded, and the actual stopping time is retained.

The same recorded trajectories supply the amplitude, frequency, and Floquet computations. To obtain the component amplitudes, we first project each state time series onto orthonormal bases of the spatial isotypic components. We take the half peak-to-peak range in each projected coordinate and combine these ranges by their Euclidean norm for the three-dimensional components.
Figure \ref{fig:bif-diagram} shows the FFT of the representative post-bifurcation trajectory at $\alpha=\alpha_*+0.01$, together with the predicted limiting angular frequency $\beta_{1,1}$.

\subsection{Floquet stability computations}\label{sec:floquet}
To examine stability beyond the $V_1$ subspace, we compute Floquet multipliers of the two-cluster periodic orbit. Write
$\bm x_{\mathrm p}(t)=\bm s q(t)$, where $\bm s=(1,-1,-1,1,-1,1,1,-1)^T$ and $C\bm s=\mu_1\bm s$, and put $J(t)=\int_{t-\tau_2}^tq(s)\,ds$. Since the derivative of $\tanh$ is even, the coupling Jacobian along this orbit is $\operatorname{sech}^2(\mu_1q(t))C$. The functions $q$ and $J$ are obtained from the nonlinear auxiliary-variable simulation, in which $\dot J(t)=q(t)-q(t-\tau_2)$, and are subsequently treated as known periodic coefficients in the variational equation.

The variational equation separates along the four spatial isotypic components. For a perturbation $z_j$ on $V_j$, introduce $S_j(t)=\int_{t-\tau_2}^tz_j(s)\,ds$. The same auxiliary-variable formulation gives
\begin{equation}\label{eq:floquet-variational}
\begin{aligned}
\dot z_j(t)
&=\big[\gamma-a-\mu_j\operatorname{sech}^2(\mu_1q(t))\big]z_j(t)
-\gamma z_j(t-\tau_1)\\
&\quad-\alpha b\operatorname{sech}^2(bJ(t))S_j(t),\\
\dot S_j(t)&=z_j(t)-z_j(t-\tau_2).
\end{aligned}
\end{equation}
For each prescribed perturbation history $\phi_j$, we set $S_j(0)=\int_{-\tau_2}^0\phi_j(s)\,ds$. The scalar systems for $V_3$ and $V_4$ each occur three times, so the calculation tests all spatial directions, including those transverse to $V_1$.

The periodic reference orbit is extracted from the same terminal trajectory used for the amplitude and FFT calculations. Its period $T$ is measured between the last two successive upward zero crossings of the $V_1$ projection, using cubic interpolation. Consistency of the periods and amplitudes is checked over several late cycles. We also check transverse contamination, the equation residual, and agreement of $J$ with the integral of $q$. Periodic cubic interpolation then supplies $q$ and $J$ in \eqref{eq:floquet-variational}. Reference trajectories that fail these checks are flagged as unresolved and omitted from the Floquet plots.

We approximate the monodromy operator using nodal cubic-spline histories on a uniform grid of $N+1$ points in $[-\tau_2,0]$. The initial value of $S_j$ is evaluated exactly from each supplied spline. JiTCDDE integrates \eqref{eq:floquet-variational} with absolute and relative tolerances $10^{-10}$ and maximum step $\min(0.25,\tau_2/(2N))$, after a forward startup step of $10^{-9}$ across the initial derivative discontinuity. Sampling the returned $z_j$ history on the same grid gives an $(N+1)\times(N+1)$ matrix whose eigenvalues approximate the Floquet multipliers. The initial integral is determined by the state history and adds no independent coordinate to this matrix.

Table \ref{tab:floquet} reports the largest computed nontrivial modulus $R_j$ on each component at $N=240$, with the phase multiplier omitted on $V_1$. At the middle parameter value, refinement from $N=120$ to $N=240$ changes $R_1$ from $0.94047386$ to $0.94047388$, while the phase-multiplier error decreases from $1.8\times10^{-8}$ to $5.9\times10^{-9}$. The implementation was also checked against characteristic-root predictions at the zero equilibrium in all four components. Additional compatible perturbations in $V_0,V_3,V_4$ and mixed spatial directions return toward the periodic orbit after phase alignment. These computations provide numerical evidence of local orbital asymptotic stability, and hence periodic multiconsensus, at the sampled orbits.

\begin{table}[tbp]
\centering
\small
\begin{tabular}{c cccc c}
\hline
$\alpha-\alpha_*$ & $R_0$ & $R_1$ & $R_3$ & $R_4$
& Phase error\\
\hline
$0.005$ & 0.18709 & 0.96953 & 0.67937 & 0.37238
& $2.3\times10^{-9}$\\
$0.010$ & 0.17994 & 0.94047 & 0.66295 & 0.36007
& $5.9\times10^{-9}$\\
$0.015$ & 0.17314 & 0.91261 & 0.64718 & 0.34830
& $1.2\times10^{-9}$\\
\hline
\end{tabular}
\caption{Floquet stability of the two-cluster orbit. The phase error is the distance of the computed phase multiplier from $1$. The table gives numerical evidence at the sampled parameter values.}\label{tab:floquet}
\end{table}

\section{Conclusions and discussion}\label{sec:conclusion}
In this paper, we introduced equations for the closed-loop dynamics of a multi-agent system where agents have a continuous state memory and finite-window trend feedback. We showed sufficient conditions on the parameters for consensus at the origin, and gave theorems for global Hopf bifurcation of symmetric branches of solutions following the breakdown of consensus. 

We then formulated an economic model as one illustrative example of the type of system which could naturally be modeled by these dynamics, showed its symmetric global bifurcation, and combined numerical simulations with Floquet calculations to support local orbital stability of the resulting two-cluster oscillations at sampled parameter values. 

The Floquet calculations complement the existence and symmetry information supplied by the degree. Their conclusions concern the sampled periodic orbits; continuation of stability and secondary bifurcations along the branch remain directions for further work. 

The system studied in this paper was inspired by natural observations of human decision-making. It is intuitive and plausible that, in many situations, agents consider continuous histories of past states and trends in order to inform future decisions, especially in situations where those agents may know that their input data is noisy (or where designers of control protocols anticipate and account for this). In such situations, distributed delay terms of retarded type can add a natural smoothing, and the pseudoneutral term provides a complementary form of trend memory.

One potential extension of the results in this paper is to explicitly add some stochastic noise to the system to study these smoothing properties in more detail. Another could be the addition of transmission delays in the coupling terms, or considering the delay term as the bifurcation parameter. Our motivation for considering the strength of continuous memory as the bifurcation parameter in this system was to study the effects of the relative weight of the state-memory and trend-memory terms on the overall dynamics.

\end{document}